\documentclass[11pt,reqno,letterpaper]{amsart}
\usepackage{amssymb,mathrsfs,graphicx,extpfeil}
\usepackage{epsfig}
\usepackage{indentfirst, latexsym, amssymb, enumerate,amsmath,graphicx}
\usepackage{float}
\usepackage{colortbl}
\usepackage{epsfig,subfigure}
\usepackage{mathtools}
\usepackage[width=6in]{geometry}
 \usepackage[pagewise]{lineno}
\usepackage{amsmath}
\allowdisplaybreaks[4]
\title[The Kuramoto model with frustration on general digraph]{Emergence of synchronization in Kuramoto model with frustration under general network topology}

\author[Tingting Zhu]{Tingting Zhu \textsuperscript{$\ddag$}
}
\email{ttzhud201880016@163.com}

\newtheorem{theorem}{Theorem}[section]
\newtheorem{lemma}{Lemma}[section]
\newtheorem{corollary}{Corollary}[section]

\newtheorem{remark}{Remark}[section]

\newtheorem{definition}{Definition}[section]

\def\charf {\mbox{{\text 1}\kern-.30em {\text l}}}

\begin{document}

\date{\today}


\thanks{$^\ddag$ Key Laboratory of Applied Mathematics and Artificial Intelligence Mechanism, Hefei University, Hefei 230601, Anhui, China}
\thanks{The author appreciated the guidance and modification of Xiongtao Zhang which helped to enhance the previous manuscript.}

\begin{abstract}
In this paper, we will study the emergent behavior of Kuramoto model with frustration on a general digraph containing a spanning tree. We provide a sufficient condition for the emergence of asymptotical synchronization if the initial data is confined in half circle.  As lack of uniform coercivity in general digraph, we apply the node decomposition criteria in \cite{H-L-Z20} to capture a clear hierarchical structure, which successfully yields the dissipation mechanism of phase diameter and a small invariant set after finite time. Then the dissipation of frequency diameter will be clear, which eventually leads  to the synchronization. 

%
\end{abstract}
\maketitle \centerline{\date}

\section{Introduction}\label{sec:1}
Synchronized behavior in complex systems is ubiquitous and has been extensively investigated in various academic communities such as physics, biology, engineering \cite{B-B96,D-M08, L-P-L-S,P-R-K, P-E-G,St,S-S93,T-T}, etc. Recently, sychronization mechanism has been applied in  control of robot systems and power systems \cite{D-B12,D-C-B13,P-L-S-G-P}. The rigorous mathematical treatment of synchronization phenomena was started by two pioneers Winfree \cite{Wi} and Kuramoto \cite{Ku,Ku84} several decades ago, who introduced different types of first-order systems of ordinary differential equations to describe the synchronous behaviors. These models contain rich emergent behaviors such as synchronization, partially phase-lcoking and nonlinear stability, etc., and have been extensively studied in both theoretical and numerical level  \cite{ B-S00, C-H-J-K12,C-S09,D-B11,H-H-K,H-K-K-Z,H-K-L-N21,H-K-P15,H-K-R,H-R20,M-S07,St}.   

In this paper, we address the synchronous problem of Kuramoto model on a general graph under the effect of frustration. To fix the idea, we consider a digraph $\mathcal{G} = (\mathcal{V}, \mathcal{E})$ consisting of a finite set $\mathcal{V} = \{1,\ldots, N\}$ of vertices and a set $\mathcal{E} \subset \mathcal{V} \times \mathcal{V}$ of directed arcs. We assume that Kuramoto oscillators are located at vertices and interact with each other via the underlying network topology. For each vertex $i$, we denote the set of its neighbors by $\mathcal{N}_i$, which is the set of vertices that directly influence vertex $i$. Now, let $\theta_i = \theta_i(t)$ be the phase of the Kuramoto oscillator at vertex $i$, and define the $(0,1)$-adjacency matrix $(\chi_{ij})$ as follows:
\begin{equation*}
\chi_{ij} = 
\begin{cases}
\displaystyle 1 \quad \mbox{if the $j$th oscillator influences the $i$th oscillator}, \\
\displaystyle 0 \quad \mbox{otherwise}.
\end{cases}
\end{equation*}
Then, the set of neighbors of $i$-th oscillator is actually $\mathcal{N}_i := \{j: \chi_{ij} >0\}$.
In this setting, the Kuramoto model with frustration on a general network is governed by the following ODE system:
\begin{equation}\label{KuM}
\begin{cases}
\displaystyle \dot{\theta}_i(t) = \Omega_i + K \underset{j \in \mathcal{N}_i}{\sum} \sin(\theta_j(t) - \theta_i(t) + \alpha), \quad t>0, \quad i \in \mathcal{V}, \\
\displaystyle \theta_i(0) = \theta_{i0},
\end{cases}
\end{equation}
where $\Omega_i, K, N$ and $\alpha \in (0, \frac{\pi}{2})$ are the natural frequency of the $i$th oscillator, coupling strength, the number of oscillators and the uniform frustration between oscillators, respectively. For the case of nonpositive frustration, we can reformulate such a system into \eqref{KuM} form by taking $\hat{\theta}_i = - \theta_i$ for $i =1,2,\ldots, N$. Note that the well-posedness of system \eqref{KuM} is guaranteed by the standard Cauchy-Lipschitz theory since the vector field on the R.H.S of \eqref{KuM} is analytic. 

Comparing to the original Kuramoto model, there are two additional structures, i.e., frustration and general digraph. The frustration was introduced by Sakaguchi and Kuramoto \cite{S-K86}, due to the observation that a pair of strongly coupled oscillators eventually oscillate with a common frequency that deviates from the average of their natural frequencies. On the other hand, the original all-to-all symmetric network is an ideal setting, thus it is natural to further consider general digraph case. Therefore, the frustration model with general digraph is more realistic in some sense. Moreover, these two structures also lead to richer phenomenon. For instance, the author in \cite{D92} observed that the frustration is common in disordered interactions, and the author in \cite{Z11} found that frustration can induce the desynchronization through varying the value of $\alpha$ in numerical simulations. For more information, please refer to \cite{C-L18, D-X, Ha-K-L14,H-L14, H-L-X, H-L-Z20, L11,O-C-K-K08,P-R-C98,T-A-A09}.

%

However, mathematically, for the Kuramoto model, the frustration and general digraph structures generate a lot of difficulties in rigorous analysis. For instance, the conservation law and gradient flow structure are lost, and thus the asymptotic states and dissipation mechanism become non-trivial. For all-to-all and symmetric case with frustration, in \cite{H-K-L14}, the authors provided sufficient frameworks leading to complete synchronization under the effect of uniform frustration. In their work, they required initial configuration to be confined in half circle. Furthermore, the authors in \cite{L-H16} dealt with the stability and uniqueness of emergent phase-locked states. In particular, the authors in \cite{H-K-Z18} exploited order parameter approach to study the identical Kuramoto oscillators with frustration. They showed that an initial configuration whose order parameter is bounded below will evolve to the complete phase synchronization or the bipolar state exponentially fast. On the other hand, for non-all-to-all case without frustration, the authors in \cite{D-H-K20} lifted the Kuramoto model to second-order system such that the second-order formulation enjoys several similar mathematical structures to that of Cucker-Smale flocking model \cite{D-H-K19}. But this method only works when the size of initial phases is less than a quarter circle, as we know the cosine function becomes negative if $\frac{\pi}{2}< \theta<\pi$. To the best knowledge of the authors, there is few work on the Kuramoto model over general digraph with frustration. The authors in \cite{H-K-P18} studied the Kuramoto model with frustrations on a complete graph which is a small perturbation of all-to-all network, and provided synchronization estimates in half circle. 


Our interest in this paper is studying the system \eqref{KuM} with uniform frustration on a general digraph. As far as the authors know, when the ensemble is distributed in half circle, the dissipation structure of the Kuramoto model with general digragh is still unclear. The main difficulties comes from the loss of uniform coercive inequality, which is due to the non-all-to-all and non-symmetric interactions. Thus we cannot expect to capture the dissipation from Gronwall-type inequality of phase diameter. For example, the time derivative of the phase diameter may be zero at some time for general digraph case. To this end, we switch to follow similar idea in \cite{H-L-Z20} to gain the dissipation through hypo-coercivity. Different from \cite{H-L-Z20} which deals with the Cucker-Smale model on a general digraph, the interactions in Kuramoto model requires more delicate estimates due to the lack of monotonicity of sine function in half circle. Eventually, we have the following main theorem..

%

\begin{theorem}\label{enter_small}
Suppose the network topology $(\chi_{ij})$ contains a spanning tree, $D^\infty$ is a given positive constant such that $D^\infty <\frac{\pi}{2}$, and all the oscillators are initially confined in half circle, i.e.,
\[D(\theta(0)) < \pi.\]
Then for sufficient large coupling strength $K$ and small frustration $\alpha$, there exists a finite time $t_*>0$  such that
\[D(\theta(t)) \le D^\infty, \quad \forall t \in [t_*, \infty).\]

%
\end{theorem}

\begin{remark}
Theorem \ref{enter_small} claims that all oscillators confined initially in half circle will enter a small region after some finite time. It is natural to ask how large $K$ and how small $\alpha$ we need to guarantee the Theorem \ref{enter_small}. In fact, according to the proof in later sections, we have the following explicit constraints on $K$ and $\alpha$,
\begin{equation}\label{condition_2}
\begin{aligned}
&\tan \alpha < \frac{1}{\left(1+ \frac{(d+1)\zeta}{\zeta - D(\theta(0))}\right)2Nc} \frac{\beta^{d+1}D^\infty}{[4(2N+1)c]^d}, \quad D^\infty + \alpha < \frac{\pi}{2},\\
&1 > \left(1+ \frac{(d+1)\zeta}{\zeta - D(\theta(0))}\right) \frac{c[4(2N+1)c]^d}{\beta^{d+1}D^\infty} \left(\frac{D(\Omega)}{K\cos \alpha} + \frac{2N \sin \alpha }{\cos \alpha}\right),
\end{aligned}
\end{equation}
where $d$ is the number of general nodes which is smaller than $N$ (see Definition \ref{node}), $D(\Omega)$ is the diameter of natural frequency, and the other parameters $\zeta$, $\gamma$, $\eta$, $\beta$ and $c$ are positive constants which satisfy the following properties,
\begin{equation}\label{condition_1}
\begin{aligned}
&D(\theta(0)) < \zeta < \gamma < \pi, \quad \eta > \max \left\{\frac{1}{\sin \gamma}, \frac{2}{1 - \frac{\zeta}{\gamma}}\right\},\\ 
&\beta=1-\frac{2}{\eta}, \quad c = \frac{\left(\sum_{j=1}^{N-1}\eta^j A(2N,j) + 1\right)\gamma}{\sin \gamma},
\end{aligned}
\end{equation}
where $A(2N,j)$ denotes the permutation. It's obvious that we can find admissible parameters satisfying \eqref{condition_1} since $D(\theta(0)) < \pi$. Once the parameters are fixed, we immediately conclude \eqref{condition_2} holds for small $\alpha$ and large $K$.
\end{remark}
\begin{remark}
After $t_*$, all oscillators are confined in a small region less than $\frac{\pi}{2}$, and Kuramoto model \eqref{KuM} will be equivalent to Cucker-Smale type model with frustration (see \eqref{s_KuM}). Therefore, we can directly apply the methods and results in \cite{D-H-K20} to conclude the emergence of frequency synchronization for large coupling and small frustration (see Corollary \ref{complete_syn}). Therefore, to guarantee the emergence of synchronization, it suffices to show the detailed proof of Theorem \ref{enter_small}.
\end{remark}

The rest of the paper is organized as follows. In Section \ref{sec:2}, we recall some concepts on the network topology and provide an a priori local-in-time estimate on the phase diameter of whole ensemble with frustration. In Section \ref{sec:3},  we consider a strong connected ensemble with frustration for which the initial phases are distributed in a half circle. We show that for large coupling strength and small frustration, the phase diameter plus a phase shift is uniformly bounded by a small value after some finite time. In Section \ref{sec:4}, we study the general network with a spanning tree structure under the effect of uniform frustration. We use the inductive argument and show that Kuramoto oscillators will concentrate into a small region less than a quarter circle in finite time, which eventually leads to the emergence of synchronization exponentially fast. Section \ref{sec:5} is devoted to a brief summary.\newline

\section{Preliminaries}\label{sec:2}
\setcounter{equation}{0}
In this section, we first introduce some fundamental concepts such as spanning tree and node decomposition of a general network \eqref{KuM}. Then, we will provide some necessary notations and an a priori estimate that will be frequently used in later sections.

\subsection{Directed graph}

Let the network topology be registered by the neighbor set $\mathcal{N}_i$ which consists of all neighbors of the $i$th oscillator. Then, for a given set of $\{\mathcal{N}_i\}_{i=1}^N$ in system \eqref{KuM}, we have the following definition.

\begin{definition}\label{spanning_tree}
\begin{enumerate}[(1)]
\item The Kuramoto digraph $\mathcal{G} = (\mathcal{V}, \mathcal{E})$ associated to \eqref{KuM} consists of a finite set $\mathcal{V} = \{1,2,\ldots,N\}$ of vertices, and a set $\mathcal{E} \subset \mathcal{V} \times \mathcal{V}$ of arcs with ordered pair $(j,i) \in \mathcal{E}$ if $j \in \mathcal{N}_i$.

\item A path in $\mathcal{G}$ from $i_1$ to $i_k$ is a sequence $i_1,i_2,\ldots,i_k$ such that
\[i_s \in \mathcal{N}_{i_{s+1}} \quad \mbox{for} \ 1 \le s \le k-1.\]
If there exists a path from $j$ to $i$, then vertex $i$ is said to be reachable from vertex $j$.

\item The Kuramoto digraph contains a spanning tree if we can find a vertex such that any other vertex of $\mathcal{G}$ is reachable from it.
\end{enumerate}
\end{definition}

\noindent In order to guarantee the emergence of synchronization, we will always assume the existence of a spanning tree throughout the paper. Now we recall the concepts of root and general root introduced in \cite{H-L-Z20}. Let $l,k \in \mathbb{N}$ with $1 \le l \le k \le N$, and let $C_{l,k} = (c_l,c_{l+1},\ldots,c_k)$ be a vector in $\mathbb{R}^{k-l+1}$ such that
\[c_i \ge 0, \quad l \le i \le k \quad \mbox{and} \quad \sum_{i=l}^k c_i = 1.\]
For an ensembel of $N$-oscillators with phases $\{\theta_i\}_{i=1}^N$, we set $\mathcal{L}_l^k(C_{l,k})$ to be a convex combination of $\{\theta_i\}_{i=l}^k$ with the coefficient $C_{l,k}$:
\[\mathcal{L}_l^k(C_{l,k}) := \sum_{i=l}^k c_i \theta_i.\]
Note that each $\theta_i$ is a convex combination of itself, and particularly $\theta_N = \mathcal{L}_N^N(1)$ and $\theta_1 = \mathcal{L}_1^1(1)$.

\begin{definition}\label{general_root}
(Root and general root)
\begin{enumerate}
\item We say $\theta_k$ is a root if it is not affected by the rest oscillators, i.e., $j \notin \mathcal{N}_k$ for any $j \in \{1,2,\ldots,N\} \setminus \{k\}$.

\item We say $\mathcal{L}_l^k(C_{l,k})$ is a general root if $\mathcal{L}_l^k(C_{l,k})$ is not affected by the rest oscillators, i.e., for any $i\in \{l, l+1,\ldots,k\}$ and $j \in \{1,2,\ldots,N\} \setminus \{l,l+1,\ldots, k\}$, we have $j \notin \mathcal{N}_i$. 
\end{enumerate}
\end{definition}

\begin{lemma}\cite{H-L-Z20}
The following assertions hold.
\begin{enumerate}
\item If the network contains a spanning tree, then there is at most one root.
\item Assume the network contains a spanning tree. If $\mathcal{L}_k^N(C_{k,N})$ is a general root, then $\mathcal{L}_1^l(C_{1,l})$ is not a general root for each $l \in \{1,2,\ldots,k-1\}$.
\end{enumerate}
\end{lemma}

\subsection{Node decomposition}
In this part, we will recall the concept of maximum node introduced in \cite{H-L-Z20}. Then, we can follow node decomposition introduced in \cite{H-L-Z20} to represent the whole graph $\mathcal{G}$ (or say vertex set $\mathcal{V}$) as a disjoint union of a sequence of nodes. The key point is that the node decomposition shows a hierarchical structure, then we can exploit this advantage to apply the induction principle. Let $\mathcal{G} = (\mathcal{V}, \mathcal{E}), \mathcal{V}_1 \subset \mathcal{V}$, and a subgraph $\mathcal{G}_1 = (\mathcal{V}_1, \mathcal{E}_1)$ is the digraph with  vertex set $\mathcal{V}_1$ and arc set $\mathcal{E}_1$ which consists of the arcs in $\mathcal{G}$ connecting agents in $\mathcal{V}_1$. For a given digraph $\mathcal{G} = (\mathcal{V}, \mathcal{E})$, we will identify a subgraph $\mathcal{G}_1 = (\mathcal{V}_1, \mathcal{E}_1)$ with its vertex set $\mathcal{V}_1$ for convenience. Now we first present the definition of nodes below.

\begin{definition}\cite{H-L-Z20}\label{node} (Node)
Let $\mathcal{G}$ be a digraph. A subset $\mathcal{G}_1$ of vertices is called a node if it is strongly connected, i.e., for any subset $\mathcal{G}_2$ of $\mathcal{G}_1$, $\mathcal{G}_2$ is affected by $\mathcal{G}_1\setminus\mathcal{G}_2$. Moreover, if $\mathcal{G}_1$ is not affected by $\mathcal{G}\setminus\mathcal{G}_1$, we say $\mathcal{G}_1$ is a maximum node. 
\end{definition}

Intuitively, a node can be understood through a way that a set of oscillators can be viewed as a "large" oscillator. The concept of node can be exploited to simplify the structure of the digraph, which indeed helps us to catch the attraction effect more clearly in the network topology.

\begin{lemma}\label{one_maximum}\cite{H-L-Z20}
Any digraph $\mathcal{G}$ contains at least one maximum node. A digraph $\mathcal{G}$ contains a unique maximum node if and only if $\mathcal{G}$ has a spanning tree.
\end{lemma}

\begin{lemma}\label{Node decomposition}\cite{H-L-Z20}(Node decomposition)
Let $\mathcal{G}$ be any digraph. Then we can decompose $\mathcal{G}$ to be a union as $\mathcal{G} = \bigcup_{i=0}^d (\bigcup_{j=1}^{k_i} \mathcal{G}_i^j)$ such that
\begin{enumerate}
\item $\mathcal{G}_0^j$ are the maximum nodes of $\mathcal{G}$, where $1\le j\le k_0$.
\item For any $p,q$ where $1 \le p\le d$ and $1\le q \le k_p$, $\mathcal{G}_p^q$ are the maximum nodes of $\mathcal{G} \setminus (\bigcup_{i=0}^{p-1} (\bigcup_{j=1}^{k_i} \mathcal{G}_i^j))$.
\end{enumerate}
\end{lemma}

\begin{remark}\label{w.l.o.g.} Lemma \ref{Node decomposition} shows a clear hierarchical structure on a general digraph. For the convenience of later analysis, we make some comments on important notations and properties that are used throughout the paper.
\begin{enumerate}
\item From the definition of maximum node, for $1\le q \ne q' \le k_p$, we see that $\mathcal{G}_p^q$ and $\mathcal{G}_p^{q'}$ do not affect each other. Actually, $\mathcal{G}_p^q$ will only be affected by $\mathcal{G}_0$ and $\mathcal{G}_i^j$, where $1 \le i \le p-1, \ 1 \le j \le k_i$. Thus in the proof of our main theorem (see Theorem \ref{enter_small}), without loss of generality, we may assume $k_i = 1$ for all $1 \le i \le d$. Hence, the decomposition can be further simplified and expressed by
\[\mathcal{G} = \bigcup_{i=0}^d \mathcal{G}_i,\]
where $\mathcal{G}_p$ is a maximum node of $\mathcal{G}\setminus (\bigcup_{i=0}^{p-1} \mathcal{G}_i)$.

\item For a given oscillator $\theta_i^{k+1} \in \mathcal{G}_{k+1}$, we denote by $\bigcup_{j=0}^{k+1} \mathcal{N}_i^{k+1}(j)$ the set of neighbors of $\theta_i^{k+1}$, where $\mathcal{N}_i^{k+1}(j)$ represents the neighbors of $\theta_i^{k+1}$ in $\mathcal{G}_j$. Note that the node decomposition and spanning tree structure in $\mathcal{G}$ guarantee that $\bigcup_{j=0}^k \mathcal{N}_i^{k+1}(j) \ne \emptyset$.
\end{enumerate}
\end{remark}

\subsection{Notations and local estimates} In this part, for notational simplicity, we introduce some notations, such as the extreme phase, phase diameter of $\mathcal{G}$ and the first $k+1$ nodes, natural frequency diameter, and cardinality of subdigraph:
\begin{align*}
&\theta_M =\max_{1\le k \le N} \{\theta_k\} = \max_{0\le i \le d} \max_{1 \le j \le N_i} \{\theta_j^i\}, \quad \theta_m = \min_{1\le k \le N} \{\theta_k\} = \min_{0\le i \le d} \min_{1 \le j \le N_i} \{\theta_j^i\},\\
&D(\theta) = \theta_M - \theta_m,\quad D_k(\theta) = \max_{0 \le i \le k} \max_{1 \le j \le N_i} \{\theta_j^i\} - \min_{0 \le i \le k} \min_{1 \le j \le N_i} \{\theta_j^i\}, \\
&\Omega_M =\max_{0\le i \le d} \max_{1 \le j \le N_i} \{\Omega_j^i\}, \quad \Omega_m =\min_{0\le i \le d} \min_{1 \le j \le N_i} \{\Omega_j^i\}, \quad D(\Omega) = \Omega_M - \Omega_m,\\
&N_i = |\mathcal{G}_i|, \quad S_k = \sum_{i=0}^k N_i, \quad 0 \le k \le d, \quad \sum_{i=0}^d N_i = N.
\end{align*}
Finally, we provide an a priori local-in-time estimate on the phase diameter to finish this section, which states that all oscillators can be confined in half circle in short time.
\begin{lemma}\label{diameter_zeta}
Let $\theta_i$ be a solution to system \eqref{KuM} and suppose the initial phase diameter satisfies
\[D(\theta(0)) < \zeta < \gamma < \pi,\]
then there exists a finite time $\bar{t} >0$ such that the phase diameter of whole ensemble remains less than $\zeta$ before $\bar{t}$, i.e.,
\begin{equation*}
D(\theta(t)) < \zeta, \quad t \in [0,\bar{t}) \quad \text{where} \quad \bar{t} = \frac{\zeta -D(\theta(0))}{D(\Omega) + 2NK\sin \alpha}.
\end{equation*}
\end{lemma}
\begin{proof}
From system \eqref{KuM}, we see that the dynamics of extreme phases is given by the following equations
\begin{equation*}
\dot{\theta}_M = \Omega_M + K \sum_{j \in \mathcal{N}_M} \sin (\theta_j - \theta_M + \alpha), \qquad \dot{\theta}_m = \Omega_m + K\sum_{j\in \mathcal{N}_m} \sin (\theta_j - \theta_m + \alpha).
\end{equation*}
We combine the above equations to estimate the dynamics of phase diameter
\begin{equation}\label{A-1}
\begin{aligned}
\dot{D}(\theta) &= \dot{\theta}_M - \dot{\theta}_m  \\
&= \Omega_M - \Omega_m + K \sum_{j \in \mathcal{N}_M} \sin (\theta_j - \theta_M + \alpha) - K\sum_{j\in \mathcal{N}_m} \sin (\theta_j - \theta_m + \alpha) \\
&\le D(\Omega) + K \sum_{j \in \mathcal{N}_M} \left[\sin(\theta_j - \theta_M)\cos \alpha + \cos(\theta_j - \theta_M) \sin \alpha\right] \\
&- K\sum_{j\in \mathcal{N}_m} \left[\sin(\theta_j -\theta_m)\cos \alpha + \cos(\theta_j - \theta_m) \sin \alpha\right]\\
&= D(\Omega) + K \cos \alpha \left(\sum_{j \in \mathcal{N}_M} \sin(\theta_j - \theta_M) - \sum_{j\in \mathcal{N}_m} \sin (\theta_j - \theta_m)\right)\\
&+ K \sin \alpha \left(\sum_{j \in \mathcal{N}_M} \cos(\theta_j -\theta_M) - \sum_{j \in \mathcal{N}_m}\cos(\theta_j - \theta_m)\right),\end{aligned}
\end{equation}
where we use the formula
\[\sin (x + y) = \sin x \cos y + \cos x \sin y\]
When the phase diameter satisfies $D(\theta(t)) \le \zeta < \pi$, it is obvious that
\begin{equation*}
\sin (\theta_j -\theta_M) \le 0, \ j \in \mathcal{N}_M \quad \text{and} \quad \sin (\theta_j - \theta_m) \ge 0, \ j\in \mathcal{N}_m.
\end{equation*}
Then we see from \eqref{A-1} that
\begin{equation}\label{A-2}
\dot{D}(\theta) \le D(\Omega) + 2NK \sin \alpha .
\end{equation}
That is to say,  when $D(\theta(t)) \le \zeta$, the growth of phase diameter is no greater than the linear growth with positive slope $D(\Omega) + 2NK \sin \alpha$. Now we integrate on both sides of \eqref{A-2} from $0$ to $t$ to have
\[D(\theta(t)) \le D(\theta(0)) + (D(\Omega) + 2NK\sin \alpha)t.\]
Therefore, it yields that there exists a finite time $\bar{t}>0$ such that
\begin{equation*}
D(\theta(t)) < \zeta, \quad \forall \ t \in [0, \bar{t}),
\end{equation*}
where $\bar{t}$ is given as below
\[\bar{t} = \frac{\zeta - D(\theta(0))}{D(\Omega) + 2NK \sin \alpha}.\]
\end{proof}

\section{Strong connected case}\label{sec:3}
\setcounter{equation}{0}
We will first study the special case, i.e., the network is strongly connected. Without loss of generality, we denote by $\mathcal{G}_0$ the strong connected digraph. According to Definition \ref{node}, Lemma \ref{one_maximum} and Lemma \ref{Node decomposition}, the strong connected network consists of only one maximum node. Then in the present special case, we will show the emergence of complete synchronization. We now introduce an algorithm to construct a proper convex combination of oscillators, which can involve the dissipation from interaction of general network. More precisely, the algorithm for $\mathcal{G}_0$ consists of the following three steps:\newline

\noindent\textbf{Step 1.} For any given time $t$, we reorder the oscillator indexes to make the oscillator phases from minimum to maximum. More specifically, by relabeling the agents at time $t$, we set
\begin{equation}\label{well_order}
\theta_1^0(t) \le \theta_2^0(t) \le \ldots \le \theta_{N_0}^0(t).
\end{equation}
In order to introduce the following steps, we first provide the process of iterations for $\bar{\mathcal{L}}_k^{N_0}(\bar{C}_{k,N_0})$ and $\underline{\mathcal{L}}_1^l(\underline{C}_{1,l})$ as follows:\newline

\noindent $\bullet$($\mathcal{A}_1$): If $\bar{\mathcal{L}}_k^{N_0}(\bar{C}_{k,N_0})$ is not a general root, then we construct
\[\bar{\mathcal{L}}_{k-1}^{N_0}(\bar{C}_{k-1,N_0}) = \frac{\bar{a}_{k-1} \bar{\mathcal{L}}_k^{N_0}(\bar{C}_{k,N_0}) + \theta^0_{k-1}}{\bar{a}_{k-1} + 1}.\]

\noindent $\bullet$($\mathcal{A}_2$): If $\underline{\mathcal{L}}_1^l(\underline{C}_{1,l})$ is not a general root, then we construct
\[\underline{\mathcal{L}}_1^{l+1}(\underline{C}_{1,l+1}) = \frac{\underline{a}_{l+1} \underline{\mathcal{L}}_1^l(\underline{C}_{1,l}) + \theta^0_{l+1}}{\underline{a}_{l+1} + 1}\]

\noindent\textbf{Step 2.} According to the strong connectivity of $\mathcal{G}_0$, we immediately know that $\bar{\mathcal{L}}_1^{N_0}(\bar{C}_{1,N_0})$ is a general root, and $\bar{\mathcal{L}}_k^{N_0}(\bar{C}_{k,N_0})$ is not a general root for $k >1$. Therefore, we may start from $\theta^0_{N_0}$ and follow the process $\mathcal{A}_1$ to construct $\bar{\mathcal{L}}_k^{N_0}(\bar{C}_{k,N_0})$ until $k=1$.\newline

\noindent\textbf{Step 3.} Similarly, we know that $\underline{\mathcal{L}}_1^{N_0}(\underline{C}_{1,N_0})$ is a general root and $\underline{\mathcal{L}}_1^l(\underline{C}_{1,l})$ is not a general root for $l < N_0$. Therefore, we may start from $\theta^0_1$ and follow the process $\mathcal{A}_2$ until $l=N_0$.\newline  

It is worth emphasizing that the order of the oscillators may change along time $t$, but at each time $t$, the above algorithm does work. For convenience, the algorithm from Step 1 to Step 3 will be referred to as Algorithm $\mathcal{A}$. Then, based on Algorithm $\mathcal{A}$, we will provide a priori estimates on a monotone property about the sine function, which will be crucially used later.

\begin{lemma}\label{eta_sin_inequality}
Let $\theta_i = \{\theta^0_i\}$ be a solution to system \eqref{KuM} with srong connected network $\mathcal{G}_0$.  Moreover at time $t$, we also assume that the oscillators are well-ordered as \eqref{well_order}, and the phase diameter and quantity $\eta$ satisfiy the following conditions:
\begin{equation*}
D_0(\theta(t)) < \gamma, \quad \eta > \max \left\{\frac{1}{\sin \gamma}, \frac{2}{1 - \frac{\zeta}{\gamma}} \right\},
\end{equation*}
where $\zeta, \gamma$ are given in the condition \eqref{condition_1}.
Then at time $t$, the following assertions hold
 \begin{equation*}
 \begin{cases}
 \displaystyle \sum_{i=n}^{N_0} ( \eta^{i-n} \underset{j\le i}{\min_{j \in \mathcal{N}^0_i(0)}} \sin (\theta_j^0 - \theta_i^0)) \le \sin(\theta^0_{\bar{k}_n} - \theta^0_{N_0}), \ \bar{k}_n = \min_{j \in \cup_{i=n}^{N_0} \mathcal{N}^0_i(0)} j, \ 1\le n \le N_0, \\
 \displaystyle \sum_{i=1}^n (\eta^{n-i} \underset{j \ge i}{\max_{j \in \mathcal{N}^0_i(0)}} \sin (\theta_j^0 - \theta_i^0)) \ge \sin (\theta^0_{\underline{k}_n} - \theta_1^0), \ \underline{k}_n = \max_{j \in \cup_{i=1}^n \mathcal{N}^0_i(0)} j, \ 1 \le n \le N_0.
 \end{cases}
 \end{equation*}
\end{lemma}
\begin{proof}
For the proof of this lemma, please see \cite{Z-Z21} for details.
\newline
\end{proof}

Recall the strongly connected ensemble $\mathcal{G}_0$, and denote by $\theta_i^0\  (i=1,2,\ldots,N_0)$ the members in $\mathcal{G}_0$. For the oscillators in $\mathcal{G}_0$, based on a priori estimates in Lemma \ref{eta_sin_inequality}, we will design proper coefficients of convex combination which helps us to capture the dissipation structure. Now we assume that at time $t$, the oscillators in $\mathcal{G}_0$ are well-ordered as follows, 
\begin{equation*}
\theta^0_1(t) \le \theta^0_2(t) \le \ldots \le \theta^0_{N_0}(t).
\end{equation*}
Then we apply the process $\mathcal{A}_1$ from $\theta^0_{N_0}$ to $\theta^0_1$ and the process $\mathcal{A}_2$ from $\theta^0_1$ to $\theta^0_{N_0}$ to respectively construct 
\begin{equation}\label{coeffi_0}
\begin{aligned}
&\bar{\mathcal{L}}_{k-1}^{N_0}(\bar{C}_{k-1,N_0}) \ \mbox{with} \ \bar{a}^0_{N_0} = 0, \ \bar{a}^0_{k-1} =\eta (2N_0 -k +2)(\bar{a}^0_k + 1), \quad 2 \le k \le N_0,\\
&\underline{\mathcal{L}}_1^{k+1}(\underline{C}_{1,k+1}) \ \mbox{with} \ \underline{a}^0_1 = 0, \ \underline{a}^0_{k+1} = \eta(k+1+N_0)(\underline{a}^0_k + 1), \quad 1\le k \le N_0-1,
\end{aligned}
\end{equation}
where $N_0$ is the cardinality of $\mathcal{G}_0$ and $\eta$ is given in the condition \eqref{condition_1}. By induction criteria, we can deduce explict expressions about the constructed coefficients:
\begin{equation}\label{permutation_0}
\begin{aligned}
&\bar{a}^0_{k-1} = \sum_{j=1}^{N_0 - k+1} \eta^j A(2N_0-k+2,j), \quad 2 \le k \le N_0, \\
&\underline{a}^0_{k+1} = \sum_{j=1}^k \eta^j A(k+1+ N_0,j), \quad 1 \le k \le N_0-1.
\end{aligned}
\end{equation}
Note that $\bar{a}^0_{N_0+1-i} = \underline{a}^0_i, \ i=1,2\ldots,N_0$. And we set
\begin{equation}\label{abbreviation_0}
\bar{\theta}^0_k := \bar{\mathcal{L}}_k^{N_0}(\bar{C}_{k,N_0}),\quad \underline{\theta}^0_k := \underline{\mathcal{L}}_1^k(\underline{C}_{1,k}), \quad 1 \le k \le N_0.
\end{equation}
We define a non-negative quantity $Q^0 = \bar{\theta}_0 - \underline{\theta}_0$ where $\bar{\theta}_0 = \bar{\theta}^0_1$ and $\underline{\theta}_0 = \underline{\theta}^0_{N_0}$. Note that $Q^0(t)$ is Lipschitz continuous with respect to $t$.
We then establish the comparison relation between $Q^0$ and the phase diameter $D_0(\theta)$ of  $\mathcal{G}_0$ in the following lemma.

\begin{lemma}\label{beta_QD}
Let $\theta_i=\{\theta^0_i\}$ be a solution to system \eqref{KuM} with strong connected digraph $\mathcal{G}_0$. Assume that for the group $\mathcal{G}_0$, the coefficients $\bar{a}_k^0$'s and $\underline{a}_k^0$'s satisfy the scheme \eqref{coeffi_0}. Then at each time $t$, we have the following relation
\begin{equation*}
\beta D_0(\theta(t)) \le Q^0(t) \le D_0(\theta(t)),  \quad \beta = 1 - \frac{2}{\eta},
\end{equation*}
where $\eta$ satisfies the condition \eqref{condition_1}. 
\end{lemma}
\begin{proof}
As we choose the same design for coefficients of convex combination as that in \cite{Z-Z21}, the proof of this lemma is same as that in \cite{Z-Z21}, please see \cite{Z-Z21} for details.\newline
\end{proof}

From Lemma \ref{beta_QD}, we see that the quantity $Q^0$ can control the phase diameter $D_0(\theta)$, which play a key role in analysing the bound of phase diameter. Based on Algorithm $\mathcal{A}$ and Lemma \ref{eta_sin_inequality}, we first study the dynamics of the constructed quantity $Q^0$.

\begin{lemma}\label{dynamics_Q0}
Let $\theta_i = \{\theta^0_i\}$ be the solution to system \eqref{KuM} with strong connected digraph $\mathcal{G}_0$. Moreover, for a given sufficiently small $D^\infty < \min\left\{\frac{\pi}{2},\zeta\right\}$, assume the following conditions hold,
\begin{equation}\label{condition_3}
\begin{aligned}
&D_0(\theta(0)) < \zeta < \gamma < \pi, \quad \eta > \max \left\{\frac{1}{\sin \gamma}, \frac{2}{1 - \frac{\zeta}{\gamma}} \right\}, \\
&\tan \alpha < \frac{1}{\left(1+ \frac{\zeta}{\zeta - D(\theta(0))}\right)2N_0c} \beta D^\infty,\\
& D^\infty + \alpha < \frac{\pi}{2}, \quad K > \left(1+ \frac{\zeta}{\zeta - D_0(\theta(0)}\right) \frac{(D(\Omega) +2N_0 K\sin \alpha)c}{\cos \alpha} \frac{1}{\beta D^\infty},
\end{aligned}
\end{equation}
where $\zeta, \gamma$ are constants and 
\[c = \frac{\left(\sum_{j=1}^{N_0-1} \eta^j A(2N_0,j) + 1\right)\gamma}{\sin \gamma}.\]
Then  the phase diameter of the graph $\mathcal{G}_0$ is uniformly bounded by $\gamma$:
\[D_0(\theta(t)) < \gamma, \quad t \in [0,+\infty),\]
and the dynamics of $Q^0(t)$ is controlled by the following differential inequality
\begin{equation*}
\dot{Q}^0(t) \le D(\Omega) + 2N_0 K\sin \alpha - \frac{K\cos\alpha}{c}Q^0(t), \quad t \in [0,+\infty).
\end{equation*}
\end{lemma}
\begin{proof}
The proof is similar to \cite{Z-Z21} under the assumption that the frustration $\alpha$ is sufficiently small. However, due to the presence of frustration, there are some slight differences in the process of analysis, thus we put the detailed proof in the Appendix \ref{appendix1}.\newline
\end{proof}
Lemma \ref{dynamics_Q0} states that the phase diameter of the digraph $\mathcal{G}_0$ remains less than $\pi$ and provides the dynamics of $Q^0$. We next exploit the dynamics of $Q^0$ and find some finite time such that the phase diameter of the digraph $\mathcal{G}_0$ is uniformly bounded by a small value after the time. 
\begin{lemma}\label{diameter_0_small}
Let $\theta_i = \{\theta^0_i\}$ be a solution to system \eqref{KuM} with strong connected digraph $\mathcal{G}_0$, and suppose the assumptions in Lemma \ref{dynamics_Q0} hold. Then there exists time $t_0\ge0$ such that
\begin{equation*}
D_0(\theta(t)) \le D^\infty, \quad \mbox{for} \ t \in [t_0, +\infty),
\end{equation*}
where $t_0$ can be estimated as below and bounded by $\bar{t}$ given in Lemma \ref{diameter_zeta}
\begin{equation}\label{D-1}
t_0 < \frac{\zeta}{\frac{K\cos\alpha}{c}\beta D^\infty - (D(\Omega) + 2N_0 K \sin \alpha)} < \bar{t}.
\end{equation}
\end{lemma}
\begin{proof}
From Lemma \ref{dynamics_Q0}, we see that the dynamics of quantity $Q^0(t)$ is governed by the following inequality
\begin{equation}\label{D-a1}
\dot{Q}^0(t) \le D(\Omega) + 2N_0 K\sin \alpha - \frac{K\cos\alpha}{c} Q^0(t), \quad t \in [0,+\infty).
\end{equation}
We next show that there exists some time $t_0$ such that the quantity $Q^0$ in \eqref{D-a1} is uniformly bounded after $t_0$. There are two cases we consider separately.\newline

\noindent$\diamond$ \textbf{Case 1.} We first consider the case that $Q^0(0) > \beta D^\infty$. When $Q^0(t) \in [\beta D^\infty,Q^0(0)]$, from \eqref{C-13}, we have
\begin{equation}\label{D-2}
\begin{aligned}
\dot{Q}^0(t) &\le D(\Omega) + 2N_0 K\sin \alpha - \frac{K\cos\alpha}{c} Q^0(t) \\
&\le D(\Omega) + 2N_0 K\sin \alpha - \frac{K\cos\alpha}{c} \beta D^\infty < 0.
\end{aligned}
\end{equation}
That is to say, when $Q^0(t)$ is located in the interval $[\beta D^\infty,Q^0(0)]$, $Q^0(t)$ will keep decreasing with a rate bounded by a uniform slope. Then we can define a stopping time $t_0$ as follows,
\begin{equation*}
t_0 = \inf \{t \ge 0 \ | \ Q^0(t) \le \beta D^\infty\}.
\end{equation*}
And based on the definition of $t_0$, we see that $Q^0$ will decrease before $t_0$ and has the following property at $t_0$,
\begin{equation}\label{D-c2}
Q^0(t_0) = \beta D^\infty.
\end{equation}
Moerover, from \eqref{D-2}, it is easy to see that the stopping time $t_0$ satisfies the following upper bound estimate,
\begin{equation}\label{D-b2}
t_0 \le \frac{Q^0(0) - \beta D^\infty}{\frac{K\cos\alpha}{c}\beta D^\infty - (D(\Omega) + 2N_0K\sin \alpha)} .
\end{equation}
Now we study the upper bound of $Q^0$ on $[t_0, +\infty)$. In fact, we can apply \eqref{D-2}, \eqref{D-c2} and the same arguments in \eqref{C-14} to derive
\begin{equation}\label{D-d2}
Q^0(t) \le \beta D^\infty, \ t \in [t_0,+\infty).
\end{equation}

\noindent $\diamond$ \textbf{Case 2.} For another case that $Q^0(0) \le \beta D^\infty$. We can apply the similar analysis in \eqref{C-14} to obtain
\begin{equation}\label{D-a3}
Q^0(t) \le \beta D^\infty, \quad t \in [0,+\infty).
\end{equation}
Then in this case, we directly set $t_0=0$. \newline

Therefore, from \eqref{D-d2}, \eqref{D-a3}, and Lemma \ref{beta_QD}, we derive the upper bound of $D_0(\theta)$ on $[t_0, +\infty)$ as below
 \begin{equation}\label{D-3}
D_0(\theta(t)) \le \frac{Q^0(t)}{\beta} \le D^\infty, \quad \mbox{for} \ t \in [t_0, +\infty).
\end{equation}
On the other hand, in order to verify \eqref{D-1}, we further study $t_0$ in \eqref{D-3}. Combining \eqref{D-b2} in Case 1 and $t_0 = 0$ in Case 2, we see that
\begin{equation}\label{D-4}
t_0 < \frac{\zeta}{\frac{K\cos\alpha}{c}\beta D^\infty - (D(\Omega) + 2N_0K\sin \alpha)}.
\end{equation}
Here, we use the truth that $Q^0(0) \le D_0(\theta(0)) < \zeta$. Then from the assumption about $K$ in \eqref{condition_3}, i.e.,
\[K > \left(1+ \frac{\zeta}{\zeta - D_0(\theta(0)}\right) \frac{(D(\Omega) +2N_0 K\sin \alpha)c}{\cos \alpha} \frac{1}{\beta D^\infty}, \quad c = \frac{\left(\sum_{j=1}^{N_0-1} \eta^j A(2N_0,j) + 1\right)\gamma}{\sin \gamma},\]
it yields the following estimate about $t_0$,
\begin{equation}\label{D-5}
\begin{aligned}
t_0 &< \frac{\zeta}{(1 + \frac{\zeta}{\zeta - D_0(\theta(0))})(D(\Omega) + 2N_0K \sin \alpha) - (D(\Omega) + 2N_0K\sin \alpha)} \\
&= \frac{\zeta - D_0(\theta(0))}{D(\Omega) + 2N_0K \sin \alpha} = \bar{t}.
\end{aligned}
\end{equation}
Note that in this special strong connected case, it's clear that $N_0 = N$ and $D_0(\theta) = D(\theta)$ in Lemma \ref{diameter_zeta}. 

Thus, combining \eqref{D-3}, \eqref{D-4} and \eqref{D-5},  we derive the desired results.\newline
\end{proof}

\section{General network}\label{sec:4}
\setcounter{equation}{0}
In this section, we investigate the general network with a spanning tree structure, and prove our main result Theorem \ref{enter_small}, which states that synchronization will emerge for Kuramoto model with frustrations. According to Definition \ref{node} and Lemma \ref{one_maximum}, we see that the digraph $\mathcal{G}$ associated to system \eqref{KuM} has a unique maximum node if it contains a spanning tree structure. And From Remark \ref{w.l.o.g.}, without loss of generality,  we assume $\mathcal{G}$ is decomposed into a union as $\mathcal{G} = \bigcup_{i=0}^d \mathcal{G}_i$, where $\mathcal{G}_p$ is a maximum node of $\mathcal{G}\setminus (\bigcup_{i=0}^{p-1} \mathcal{G}_i)$.

We have studied the situation $d=0$ in Section \ref{sec:3}, and we showed that the phase diameter of the digraph $\mathcal{G}_0$ is uniformly bounded by a small value after some finite time, i.e., the oscillators of $\mathcal{G}_0$ will concentrate into a small region of quarter-circle. However, for the case that $d > 0$, $\mathcal{G}_k$'s are not maximum nodes in $\mathcal{G}$ for $k \ge 1$. Hence, the methods in Lemma \ref{dynamics_Q0} and Lemma \ref{diameter_0_small} can not be directly exploited for the situation $d>0$. More precisely, the oscillators in $\mathcal{G}_i$ with $i < k$ perform as an attraction source and affect the agents in $\mathcal{G}_k$. Thus when we study the behavior of agents in $\mathcal{G}_k$, the information from $\mathcal{G}_i$ with $i < k$ can not be ignored.

From Remark \ref{w.l.o.g.} and node decomposition, the graph $\mathcal{G}$ can be represented as 
\begin{equation*}
\mathcal{G} = \bigcup_{k=0}^d \mathcal{G}_k, \quad |\mathcal{G}_k| = N_k,
\end{equation*}
and we denote the oscillators in $\mathcal{G}_k$ by $\theta^k_i$ with $1 \le i \le N_k$. Then we assume that at time $t$, the oscillators in each $\mathcal{G}_k$ are well-ordered as below:
\begin{equation}\label{well_order_k}
\theta^k_1(t) \le \theta^k_2(t) \le \ldots \le \theta^k_{N_k}(t),  \quad 0 \le k \le d.
\end{equation}
For each subdigraph $\mathcal{G}_k$ with $k \ge 0$ which is strongly connected, we follow the process in  Algorithm $\mathcal{A}_1$ and $\mathcal{A}_2$ to construct $\bar{\mathcal{L}}_{l-1}^{N_k}(\bar{C}_{l-1,N_k})$ and $\underline{\mathcal{L}}_1^{l+1}(\underline{C}_{1,l+1})$ by redesigning the coefficients $\bar{a}^k_l$ and $\underline{a}_l^k$ of convex combination as below:

\begin{equation}\label{coeffi_k}
\begin{cases}
\displaystyle \bar{\mathcal{L}}_{l-1}^{N_k}(\bar{C}_{l-1,N_k}) \ \mbox{with} \ \bar{a}^k_{N_k} = 0, \ \bar{a}^k_{l-1} =\eta (2N -l +2)(\bar{a}^k_l + 1), \quad 2 \le l \le N_k, \\
\displaystyle \underline{\mathcal{L}}_1^{l+1}(\underline{C}_{1,l+1}) \ \mbox{with} \ \underline{a}^k_1 = 0, \ \underline{a}^k_{l+1} = \eta(l+1+2N-N_k)(\underline{a}^k_l + 1), \quad 1\le l \le N_k-1,
\end{cases}
\end{equation}
By induction principle, we deduce that
\begin{equation}\label{permutation_k}
\begin{cases}
\displaystyle \bar{a}^k_{l-1} = \sum_{j=1}^{N_k - l+1} \eta^j A(2N-l+2,j), \quad 2 \le l \le N_k, \\
\displaystyle\underline{a}^k_{l+1} = \sum_{j=1}^l \eta^j A(l+1+ 2N-N_k,j), \quad 1 \le l \le N_k-1.
\end{cases}
\end{equation}
Note that $\bar{a}^k_{N_k+1-i} = \underline{a}^k_i, \ i=1,2\ldots,N_k$. By simple calculation, we have
\begin{equation}\label{a^k_1-size}
\bar{a}^k_1 = \sum^{N_k -1}_{j=1} (\eta^j A(2N,j)),  \quad \bar{a}^k_1 \le \sum^{N -1}_{j=1} (\eta^j A(2N,j)), \quad 0 \le k\le d.
\end{equation}
And we further introduce the following notations, 
\begin{align}
&\bar{\theta}^k_l := \bar{\mathcal{L}}_l^{N_k}(\bar{C}_{l,N_k}),\quad \underline{\theta}^k_l := \underline{\mathcal{L}}_1^l(\underline{C}_{1,l}), \quad 1 \le l \le N_k, \quad 0 \le k \le d,\label{abbreviation_k}\\
&\bar{\theta}_k := \bar{\mathcal{L}}_1^{N_k}(\bar{C}_{1,N_k}), \quad \underline{\theta}_k := \underline{\mathcal{L}}_1^{N_k}(\underline{C}_{1,N_k}), \quad 0\le k \le d,\label{bar_underline_k}\\
&Q^k(t) := \max_{0 \le i\le k}\{\bar{\theta}_i\} - \min_{0\le i\le k}\{\underline{\theta}_i\},  \quad 0 \le k \le d.\label{definition_Qk}
\end{align}
Due to the analyticity of the solution, $Q^k(t)$ is Lipschitz continuous. Similar to Section \ref{sec:3}, we will first establish the comparison between the quantity $Q^k(t)$ and phase diameter $D_k(\theta(t))$ of the first $k+1$ nodes, which plays a crucial role in the later analysis.

\begin{lemma}\label{beta_QD_k}
Let $\theta_i$ be a solution to system \eqref{KuM}, and assume that the network contains a spanning tree and for each subdigraph $\mathcal{G}_k$, the coefficients $\bar{a}^k_l$ and $\underline{a}^k_l$ of convex combination in Algorithm $\mathcal{A}$ satisfy the scheme \eqref{coeffi_k}. Then at each time $t$, we have the following relation
\begin{equation*}
\beta D_k(\theta(t)) \le Q^k(t) \le D_k(\theta(t)),  \quad 0\le k \le d, \quad\beta = 1 - \frac{2}{\eta},
\end{equation*}
where $D_k(\theta) = \max\limits_{0 \le i \le k} \max\limits_{1 \le j \le N_i} \{\theta_j^i\} - \min\limits_{0 \le i \le k} \min\limits_{1 \le j \le N_i} \{\theta_j^i\}$ and $\eta$ satisfies the condition \eqref{condition_1}. 
\end{lemma}
\begin{proof}
As we adopt the same construction of coefficients of convex combination in \cite{Z-Z21} which deals with the Kuramoto model without frustration on a general network, thus for the detailed proof of this lemma, please see \cite{Z-Z21}.
\end{proof}

Now we are ready to prove our main Theorem \ref{enter_small}. To this end, we will follow similar arguments in Section \ref{sec:3} to complete the proof. Actually, we will investigate the dynamics of the constructed quantity $Q^k(t)$ that involves the influences from $\mathcal{G}_i$ with $i <k$, which yields the hypo-coercivity of the phase diameter. Applying similar arguments in Lemma \ref{dynamics_Q0} and Lemma \ref{diameter_0_small}, we have the following estimates for the first maximal node $\mathcal{G}_0$.

\begin{lemma}\label{L4.2}
Suppose that the network topology contains a spanning tree, and let $\theta_i$ be a solution to \eqref{KuM}. Moreover, assume that the initial data and the quantity $\eta$ satisfy
 \begin{equation}\label{condition_11}
 D(\theta(0)) < \zeta < \gamma < \pi, \quad \eta > \max \left\{\frac{1}{\sin \gamma}, \frac{2}{1 - \frac{\zeta}{\gamma}} \right\},
 \end{equation}
 where $\zeta, \gamma$ are positive constants. And for a given sufficiently small $D^\infty < \min\{\frac{\pi}{2},\zeta\}$, assume the frustration $\alpha$ and coupling strength $\kappa$ satisfy
\begin{equation}\label{condition_21}
\begin{aligned}
&\tan \alpha < \frac{1}{\left(1+ \frac{(d+1)\zeta}{\zeta - D(\theta(0))}\right)2Nc} \frac{\beta^{d+1}D^\infty}{[4(2N+1)c]^d}, \quad D^\infty + \alpha < \frac{\pi}{2},\\
&K > \left(1+ \frac{(d+1)\zeta}{\zeta - D(\theta(0))}\right) \frac{(D(\Omega) + 2NK \sin \alpha)c}{\cos \alpha} \frac{[4(2N+1)c]^d}{\beta^{d+1}D^\infty},
\end{aligned}
\end{equation}
 where $d$ is the number of general nodes and
\[c = \frac{(\sum_{j=1}^{N-1} \eta^j A(2N,j) + 1)\gamma}{\sin \gamma}.\]
Then the following two assertions hold for the maximum node $\mathcal{G}_0$:
\begin{enumerate}
\item  The dynamics of $Q^0(t)$ is governed by the following equation
\begin{equation*}
\dot{Q}^0(t) \le D(\Omega) + 2N K\sin \alpha - \frac{K\cos\alpha}{c} Q^0(t), \quad t \in [0,+\infty),
\end{equation*}
\item there exists time $t_0 \ge 0$ such that
\begin{equation*}
D_0(\theta(t)) \le \frac{\beta^d D^\infty}{[4(2N+1)c]^d}, \quad \mbox{for} \ t \in [t_0, +\infty),
\end{equation*}
where $t_0$ can be estimated as below and bounded by $\bar{t}$ given in Lemma \ref{diameter_zeta}
\begin{equation*}
t_0 < \frac{\zeta}{\frac{K\cos\alpha}{c}\frac{\beta^{d+1} D^\infty}{[4(2N+1)c]^d} - (D(\Omega) + 2N K \sin \alpha)} < \bar{t}.
\end{equation*}
\end{enumerate}
 \end{lemma}

Next, inspiring from Lemma \ref{L4.2}, we make the following reasonable ansatz for $Q^k(t)$ with $0 \le k \le d$.\newline 

\noindent \textbf{Ansatz:}
\begin{enumerate}
\item The dynamics of quantity $Q^k(t)$ in time interval $[0,\infty)$ is governed by the following differential inequality,
\begin{equation}\label{F-1}
\begin{aligned}
\dot{Q}^k(t) &\le D(\Omega) +2N K\sin \alpha + (2N+1)K\cos \alpha D_{k-1}(\theta(t)) -\frac{K\cos\alpha}{c} Q^k(t), \  t \in [0,+\infty).
\end{aligned}
\end{equation}
where we assume $D_{-1}(\theta(t)) = 0$.

\item there exists a finite time $t_k \ge 0$ such that, the phase diameter $D_k(\theta(t))$ of $\bigcup_{i=0}^k \mathcal{G}_i$ is uniformly bounded after $t_k$, i.e.,
\begin{equation}\label{F-2}
D_k(\theta(t)) \le \frac{\beta^{d-k}D^\infty}{[4(2N+1)c]^{d-k}}, \quad \forall \ t\in [t_k, +\infty),
\end{equation}
where $t_k$ subjects to the following estimate,
\begin{equation}\label{F-3}
t_k < \frac{(k+1)\zeta}{\frac{K\cos\alpha}{c} \frac{\beta^{d+1}D^\infty}{[4(2N+1)c]^{d}} - (D(\Omega) + 2NK\sin \alpha )} < \bar{t} = \frac{\zeta - D(\theta(0))}{D(\Omega) + 2NK\sin \alpha }.
\end{equation}
\end{enumerate}

In the subsequence, we will split the proof of the ansatz into two lemmas by induction criteria. More precisely, based on the results in Lemma \ref{L4.2} as the initial step, we suppose the ansatz holds for $ Q^k$ and $D_k(\theta)$ with $0 \le k \le d-1$, and then prove that the ansatz also holds for $Q^{k+1}$ and $D_{k+1}(\theta)$.

\begin{lemma}\label{L4.3}
Suppose the assumptions in Lemma \ref{L4.2} are fulfilled, and the ansatz in \eqref{F-1}, \eqref{F-2} and \eqref{F-3} holds for some $k$ with $0 \le k \le d-1$. Then the ansatz  \eqref{F-1} holds for $k+1$.
\end{lemma}
\begin{proof}
We will use proof by contradiction criteria to verify the ansatz for $Q^{k+1}$. To this end, define a set
\begin{equation*}
\mathcal{B}_{k+1} = \{T>0 \  : \ D_{k+1}(\theta(t)) < \gamma, \ \forall \ t \in [0,T)\}.
\end{equation*}
From Lemma \ref{diameter_zeta}, we see that
\begin{equation*}
D_{k+1}(\theta(t)) \le D(\theta(t)) < \zeta < \gamma, \quad \forall \ t \in [0, \bar{t}).
\end{equation*}
It is clear that $\bar{t} \in \mathcal{B}_{k+1}$. Thus the set $\mathcal{B}_{k+1}$ is not empty. Define $T^* = \sup \mathcal{B}_{k+1}$. We will  prove by contradiction that $T^* = +\infty$. Suppose not, i.e., $T^* < +\infty$. It is obvious that
\begin{equation}\label{F-4}
\bar{t} \le T^*, \quad D_{k+1}(\theta(t)) < \gamma, \ \forall \ t \in [0,T^*), \quad D_{k+1}(\theta(T^*)) = \gamma.
\end{equation}

As the solution to system \eqref{KuM} is analytic, in the finite time interval $[0,T^*)$, $\bar{\theta}_i$ and $\bar{\theta}_j$ either collide finite times or always stay together. Similar to the analysis in Lemma \ref{dynamics_Q0}, without loss of generality, we only consider the situation that there is no pair of $\bar{\theta}_i$ and $\bar{\theta}_j$  staying together through all period $[0,T^*)$. That means the order of $\{\bar{\theta}_i\}_{i=0}^{k+1}$ will only exchange finite times in $[0,T^*)$, so does $\{\underline{\theta}_i\}_{i=0}^{k+1}$ . Thus, we divide the time interval $[0,T^*)$ into a finite union as below
\[[0,T^*) = \bigcup_{l=1}^r J_l, \quad J_l = [t_{l-1},t_l).\]
such that in each interval $J_l$, the orders of both $\{\bar{\theta}_i\}_{i=0}^{k+1}$ and $\{\underline{\theta}_i\}_{i=0}^{k+1}$ are preseved, and the order of oscillators in each subdigraph $\mathcal{G}_i$ with $0 \le i \le k+1$ does not change. In the following, we will show the contradiction via two steps.\newline

\noindent $\star$ \textbf{Step 1.} In this step, we first verify the Ansatz \eqref{F-1} holds for $Q^{k+1}$ on $[0, T^*)$, i.e.,
\begin{equation}\label{F-a20}
\begin{aligned}
\dot{Q}^{k+1}(t) &\le D(\Omega) + 2NK\sin \alpha + (2N + 1)K \cos \alpha D_k(\theta(t)) -\frac{K\cos\alpha}{c}Q^{k+1}(t), \  \ t \in [0,T^*).
\end{aligned}
\end{equation}
As the proof is slightly different from that in \cite{Z-Z21} and rather lengthy, we put the detailed proof in Appendix \ref{appendix2}. \newline

\noindent $\star$ \textbf{Step 2.} In this step, we will study the upper bound of $Q^{k+1}$ in \eqref{F-a20} in time interval $[t_k,T^*)$, where $t_k$ is given in Ansatz $\eqref{F-2}$ for $D_k(\theta)$. For the purposes of discussion, we rewrite the equation \eqref{F-a20} as below
\begin{equation}\label{F-d20}
\begin{aligned}
\dot{Q}^{k+1}(t) &\le -\frac{K\cos\alpha}{c} \left(Q^{k+1}(t) - (2N+1)cD_k(\theta(t)) - \frac{(D(\Omega)+2NK\sin\alpha)c}{K\cos\alpha}\right), \ t \in [0,T^*).
\end{aligned}
\end{equation}
where $c$ is expressed by the following equation
\begin{equation}\label{F-dd20}
c = \frac{(\sum_{j=1}^{N-1} \eta^j A(2N,j) + 1)\gamma }{\sin \gamma}.
\end{equation}
For the term $D_k(\theta)$ in \eqref{F-d20}, under the assumption of induction criteria, the Ansatz \eqref{F-2} holds for $D_k(\theta)$, i.e., there exists time $t_k$ such that
\begin{equation}\label{F-21}
 D_k(\theta(t)) \le \frac{\beta^{d-k} D^\infty}{[4(2N+1)c]^{d-k}}, \quad \ t \in [t_k, +\infty), \quad t_k < \bar{t}.
\end{equation}
And from the condition \eqref{condition_21}, it is obvious that
\begin{equation*}
\begin{aligned}
K &> \left(1+ \frac{(d+1)\zeta}{\zeta - D(\theta(0))}\right) \frac{(D(\Omega) + 2NK \sin \alpha)c}{\cos \alpha} \frac{[4(2N+1)c]^d}{\beta^{d+1}D^\infty}\\
&> \frac{(D(\Omega) + 2NK \sin \alpha)c}{\cos \alpha} \frac{[4(2N+1)c]^d}{\beta^{d+1}D^\infty}.
\end{aligned}
\end{equation*}
This directly yields that
\begin{equation}\label{F-22}
\frac{(D(\Omega)+2NK\sin\alpha)c}{K\cos\alpha} < \frac{\beta^{d+1}D^\infty}{[4(2N+1)c]^d} < \frac{\beta^{d-k}D^\infty}{4^{d-k}[(2N+1)c]^{d-k-1}}, 
\end{equation}
where $0\le k\le d-1, \beta< 1, c>1.$
Then for the purposes of analysing the last two terms in the bracket of \eqref{F-d20}, we add the esimates in \eqref{F-21} and \eqref{F-22} to get
\begin{equation}\label{F-23}
\begin{aligned}
&(2N+1)cD_k(\theta(t)) + \frac{(D(\Omega) +2NK\sin\alpha)c}{K\cos\alpha} \\
&\le (2N+1)c\frac{\beta^{d-k} D^\infty}{[4(2N+1)c]^{d-k}} + \frac{\beta^{d-k}D^\infty}{4^{d-k}[(2N+1)c]^{d-k-1}}\\
&\le \frac{\beta^{d-k}D^\infty}{2[4(2N+1)c]^{d-k-1}} < \frac{\beta^{d-k}D^\infty}{[4(2N+1)c]^{d-k-1}},\qquad t \in [t_k, +\infty).
\end{aligned}
\end{equation}
From Lemma \ref{diameter_zeta}, we have $t_k < \bar{t} \le T^*$, thus it makes sense when we consider the time interval $[t_k, T^*)$. Now based on the above estiamte \eqref{F-23}, we apply the differential equation \eqref{F-d20} and study the upper bound of $Q^{k+1}$ on $[t_k, T^*)$. We claim that
\begin{equation}\label{F-a23}
Q^{k+1}(t) \le \max \left\{Q^{k+1}(t_k), \frac{\beta^{d-k}D^\infty}{ [4(2N+1)c]^{d-k-1}} \right\} := M_{k+1}, \quad t \in [t_k, T^*).
\end{equation}
Suppose not, then there exists some $\tilde{t} \in (t_k, T^*)$ such that $Q^{k+1}(\tilde{t}) > M_{k+1}$. We construct a set
\[\mathcal{C}_{k+1} := \{t_k \le t < \tilde{t} : Q^{k+1}(t) \le M_{k+1}\}.\]
Since $Q^{k+1}(t_k) \le M_{k+1}$, the set $\mathcal{C}_{k+1}$ is not empty. Define $t^* = \sup \mathcal{C}_{k+1}$. Then it is easy to see that
\begin{equation}\label{F-b23}
t^* < \tilde{t}, \quad Q^{k+1}(t^*)=M_{k+1}, \quad Q^{k+1}(t) > M_{k+1} \quad \mbox{for} \ t \in (t^*, \tilde{t}].
\end{equation}
From the construction of $M_{k+1}$, \eqref{F-23} and \eqref{F-b23}, it is clear that for $t \in (t^*, \tilde{t}]$
\begin{equation*}
\begin{aligned}
&-\frac{K\cos\alpha}{c} \left(Q^{k+1}(t) - (2N+1)cD_k(\theta(t)) - \frac{(D(\Omega)+2NK\sin\alpha)c}{K\cos\alpha}\right)\\
&< -\frac{K\cos\alpha}{c}\left(M_{k+1} - \frac{\beta^{d-k}D^\infty}{[4(2N+1)c]^{d-k-1}}\right) \le 0.
\end{aligned}
\end{equation*}
Wen apply the above inequality and integrate on both sides of $\eqref{F-d20}$ from $t^*$ to $\tilde{t}$ to get
\begin{equation*}
\begin{aligned}
& Q^{k+1}(\tilde{t}) - M_{k+1} \\
&= Q^{k+1}(\tilde{t}) - Q^{k+1}(t^*) \\
&\le -\int_{t^*}^{\tilde{t}} \frac{K\cos\alpha}{c} \left(Q^{k+1}(t) - (2N+1)cD_k(\theta(t)) - \frac{(D(\Omega)+2NK\sin\alpha)c}{K\cos\alpha}\right) dt < 0
\end{aligned}
\end{equation*}
which contradicts to the truth $Q^{k+1}(\tilde{t}) - M_{k+1} >0$. Thus we complete the proof of \eqref{F-a23}.\newline

\noindent $\star$ \textbf{Step 3.} In this step, we will construct a contradiction to \eqref{F-4}.
From \eqref{F-a23}, Lemma \ref{diameter_zeta} and the fact that
\begin{equation*}
\frac{\beta^{d-k}D^\infty}{ [4(2N+1)c]^{d-k-1}} < D^\infty, \quad t_k < \bar{t}, \quad Q^{k+1}(t_k) \le D_{k+1}(\theta(t_k)) \le D(\theta(t_k)) < \zeta,
\end{equation*}
we directly obtain
\begin{equation*}
Q^{k+1}(t) \le \max \left\{Q^{k+1}(t_k), \frac{\beta^{d-k}D^\infty}{ [4(2N+1)c]^{d-k-1}} \right\} < \max\left\{\zeta,D^\infty\right\} = \zeta, \quad t\in [t_k,T^*).
\end{equation*}
From Lemma \ref{beta_QD_k} and the condition \eqref{condition_11}, it yields that
\begin{equation*}
D_{k+1}(\theta(t)) \le \frac{Q^{k+1}(t)}{\beta} < \frac{\zeta}{\beta} < \gamma, \quad t \in [t_k, T^*).
\end{equation*}
Since $D_{k+1}(\theta(t))$ is continuous, we have
\begin{equation*}
D_{k+1}(\theta(T^*)) = \lim_{t \to (T^*)^-} D_{k+1}(\theta(t)) \le \frac{\zeta}{\beta} < \gamma,
\end{equation*}
which obviously contradicts to the assumption $D_{k+1}(\theta(T^*)) = \gamma$ in \eqref{F-4}. \newline

Thus, we combine all above analysis to conclude that $T^* = +\infty$, that is to say,
\begin{equation}\label{F-c23}
D_{k+1}(\theta(t)) < \gamma, \quad \forall \ t \in [0, +\infty).
\end{equation}
Then for any finite time $T>0$, we apply \eqref{F-c23} and repeat the analysis in Step 1 to obtain that the differential inequality \eqref{F-1} holds for $Q^{k+1}$ on $[0,T)$. Thus we obtain the dynamics of $Q^{k+1}$ in whole time interval as below:
\begin{equation}\label{F-24}
\begin{aligned}
\dot{Q}^{k+1}(t) &\le D(\Omega) + 2NK\sin \alpha + (2N + 1)K \cos \alpha D_k(\theta(t)) -  \frac{K\cos\alpha}{c} Q^{k+1}(t),\ t \in [0,+\infty).
\end{aligned}
\end{equation}

Therefore, we complete the proof of the Ansatz \eqref{F-1} for $Q^{k+1}$.\newline
\end{proof}

\begin{lemma}\label{L4.4}
Suppose the conditions in Lemma \ref{L4.2} are fulfilled, and the ansatz in \eqref{F-1}, \eqref{F-2} and \eqref{F-3} holds for some $k$ with $0 \le k \le d-1$. Then the ansatz  \eqref{F-2} and \eqref{F-3} holds for $k+1$.
\end{lemma}
\begin{proof}
From Lemma \ref{L4.3}, we know the dynamic of $Q^{k+1}$ is governed by \eqref{F-24}. For the purposes of discussion, we rewrite the  differential equation \eqref{F-24} and discuss it on $[t_k, + \infty)$,
\begin{equation}\label{F-25}
\dot{Q}^{k+1}(t) \le - \frac{K\cos\alpha}{c}   \left(Q^{k+1}(t) - (2N+1)cD_k(\theta(t)) - \frac{(D(\Omega)+2NK\sin\alpha)c}{K\cos\alpha}\right), \ t \in [t_k,+\infty).
\end{equation}
where $c$ is given in \eqref{F-dd20}. In the subsequence, we will apply \eqref{F-25} to find a finite time $t_{k+1}$ such that the quantity $Q^{k+1}$ in \eqref{F-25} is uniformly bounded by a small value after $t_{k+1}$. We split into two cases to discuss.\newline

\noindent$\bullet$ \textbf{Case 1.} We first consider the case that $Q^{k+1}(t_k) > \frac{\beta^{d-k}D^\infty}{ [4(2N+1)c]^{d-k-1}}$. In this case,  When $Q^{k+1}(t) \in [\frac{\beta^{d-k}D^\infty}{ [4(2N+1)c]^{d-k-1}},Q^{k+1}(t_k)]$, we combine \eqref{F-23} and \eqref{F-25} to have
\begin{equation}\label{F-d25}
\begin{aligned}
\dot{Q}^{k+1}(t) &\le - \frac{K\cos\alpha}{c} \left(\frac{\beta^{d-k}D^\infty}{ [4(2N+1)c]^{d-k-1}} - \frac{\beta^{d-k}D^\infty}{2[4(2N+1)c]^{d-k-1}}\right)\\
&=- \frac{K\cos\alpha}{c} \frac{\beta^{d-k}D^\infty}{2[4(2N+1)c]^{d-k-1}} < 0.
\end{aligned}
\end{equation}
That is to say, when $Q^{k+1}(t)$ is located in the interval $[\frac{\beta^{d-k}D^\infty}{ [4(2N+1)c]^{d-k-1}},Q^{k+1}(t_k)]$, $Q^{k+1}(t)$ will keep decreasing with a rate bounded by a uniform slope.
Therefore, we can define a stopping time $t_{k+1}$ as follows,
\[t_{k+1}=\inf \left\{t\geq t_k\ |\ Q^{k+1}(t)\le \frac{\beta^{d-k}D^\infty}{ [4(2N+1)c]^{d-k-1}}\right\}.\]
Then, based on \eqref{F-d25} and the definition of $t_{k+1}$, we see that $Q^{k+1}$ will decrease before $t_{k+1}$ and has the following property at $t_{k+1}$,
\begin{equation}\label{F-a26}
Q^{k+1}(t_{k+1}) = \frac{\beta^{d-k}D^\infty}{ (4c)^{d-k-1}}.
\end{equation}
Moreover, from \eqref{F-d25}, it yields that the stopping time $t_{k+1}$ satisfies the following upper bound estimate, 
\begin{equation}\label{F-26}
t_{k+1} \le \frac{Q^{k+1}(t_k) - \frac{\beta^{d-k}D^\infty}{ [4(2N+1)c]^{d-k-1}}}{\frac{K\cos\alpha}{c} \frac{\beta^{d-k}D^\infty}{2[4(2N+1)c]^{d-k-1}}} + t_k.
\end{equation}
Now we study the upper bound of $Q^{k+1}$ on $[t_{k+1}, +\infty)$. In fact,
we can apply \eqref{F-d25}, \eqref{F-a26} and the same arguments in \eqref{F-a23} to derive  
\begin{equation}\label{F-a31}
Q^{k+1}(t) \le \frac{\beta^{d-k}D^\infty}{ [4(2N+1)c]^{d-k-1}}, \quad t \in [t_{k+1}, +\infty).
\end{equation}
On the other hand, in order to verify \eqref{F-3}, we further study $t_{k+1}$ in \eqref{F-26}. For the first part on the right-hand side of \eqref{F-26}, from Lemma \ref{diameter_zeta} and the fact that
\begin{equation*}
Q^{k+1}(t_k) \le D_{k+1}(\theta(t_k)) \le D(\theta(t_k)) < \zeta, \quad \frac{\beta^{d-k}D^\infty}{ 2[4(2N+1)c]^{d-k-1}} > \frac{\beta^{d+1}D^\infty}{[4(2N+1)c]^d},
\end{equation*}
we have the following estimates
\begin{equation}\label{F-27}
\frac{Q^{k+1}(t_k) - \frac{\beta^{d-k}D^\infty}{ [4(2N+1)c]^{d-k-1}}}{\frac{K\cos\alpha}{c} \frac{\beta^{d-k}D^\infty}{2[4(2N+1)c]^{d-k-1}}} < \frac{\zeta}{\frac{K\cos\alpha}{c}\frac{\beta^{d+1}D^\infty}{[4(2N+1)c]^d} - (D(\Omega) + 2NK\sin\alpha)},
\end{equation}
 where the denominator on the right-hand side of above inequality is positive from  the conditions about $K$ and $\alpha$ in \eqref{condition_21}. For the term $t_k$ in \eqref{F-26},
based on the assumption \eqref{F-3} for $t_k$, we have
\begin{equation}\label{F-28}
t_k < \frac{(k+1)\zeta}{\frac{K\cos\alpha}{c}\frac{\beta^{d+1}D^\infty}{[4(2N+1)c]^d} - (D(\Omega) + 2NK\sin\alpha)} < \bar{t} = \frac{\zeta - D(\theta(0))}{D(\Omega) + 2NK\sin\alpha}.
\end{equation}
Thus it yields from \eqref{F-26}, \eqref{F-27} and \eqref{F-28}  that the time $t_{k+1}$ satisfies
\begin{equation}\label{F-29}
t_{k+1} < \frac{(k+2)\zeta}{\frac{K\cos\alpha}{c}\frac{\beta^{d+1}D^\infty}{[4(2N+1)c]^d} - (D(\Omega) + 2NK\sin\alpha)}.
\end{equation}
Moreover, from \eqref{condition_21}, it is easy to see that the coupling strength $K$ satisfies the following inequality
\begin{equation}\label{F-30}
\begin{aligned}
K &> \left(1+ \frac{(d+1)\zeta}{\zeta - D(\theta(0))}\right) \frac{(D(\Omega) + 2NK \sin \alpha)c}{\cos \alpha} \frac{[4(2N+1)c]^d}{\beta^{d+1}D^\infty}\\
&\ge \left(1+ \frac{(k+2)\zeta}{\zeta - D(\theta(0))}\right) \frac{(D(\Omega) + 2NK \sin \alpha)c}{\cos \alpha} \frac{[4(2N+1)c]^d}{\beta^{d+1}D^\infty}, \quad 0 \le k \le d-1.
\end{aligned}
\end{equation}
Thus we combine \eqref{F-29} and \eqref{F-30} to verify the Ansatz \eqref{F-3} for $k+1$ in the first case, i.e.,  the time $t_{k+1}$ subjects to the following estimate,
\begin{equation}\label{F-31}
t_{k+1} < \bar{t} = \frac{\zeta - D(\theta(0))}{D(\Omega) + 2NK\sin\alpha}.
\end{equation}

\noindent $\bullet$ \textbf{Case 2.} For another case that $Q^{k+1}(t_k) \le \frac{\beta^{d-k}D^\infty}{ (4c)^{d-k-1}}$. Similar to the analysis in \eqref{F-a23}, we apply \eqref{F-d25} to conclude that
\begin{equation}\label{F-b31}
Q^{k+1}(t) \le \frac{\beta^{d-k}D^\infty}{ [4(2N+1)c]^{d-k-1}}, \quad t \in [t_{k}, +\infty).
\end{equation}
In this case, we directly set $t_{k+1} = t_k$. Then, from \eqref{F-28}, it yields that the inequalities \eqref{F-29} and \eqref{F-31} also hold, which finish the verification of the Ansatz \eqref{F-3} in the second case. \newline

Finally, we are ready to verify the ansatz \eqref{F-2} for $k+1$. Actually, we can apply \eqref{F-a31}, \eqref{F-b31} and Lemma \ref{beta_QD_k} to have the upper bound of $D_{k+1}(\theta)$ on $[t_{k+1}, +\infty)$ as below
\begin{equation}\label{F-32}
\begin{aligned}
D_{k+1}(\theta(t)) \le \frac{Q^{k+1}(t)}{\beta} \le \frac{\beta^{d-k-1}D^\infty}{ [4(2N+1)c]^{d-k-1}}, \quad t \in [t_{k+1}, +\infty),
\end{aligned}
\end{equation}
Then we combine \eqref{F-29}, \eqref{F-31} and \eqref{F-32} in Case 1 and similar analysis in Case 2 to conclude that the Ansatz \eqref{F-2} and \eqref{F-3} is true for $D^{k+1}(\theta)$.\newline
\end{proof}
\vspace{0.5cm}

Now, we are ready to prove our main result.\newline

\noindent \textbf{Proof of Theorem  \ref{enter_small}:}  Combining Lemma \ref{L4.2}, Lemma \ref{L4.3} and Lemma \ref{L4.4}, we apply inductive criteria to conclude that the Ansatz \eqref{F-1} --\eqref{F-3} hold for all $0\leq k\leq d$. Then, it yields from \eqref{F-2} that there exists a finite time $t_d\ge0$ such that
\begin{equation*}
D(\theta(t)) = D_d(\theta(t)) \le D^\infty, \quad \mbox{for} \ t \in [t_d, +\infty).
\end{equation*}
Thus we derive the desired result in Theorem \ref{enter_small}. 

\begin{remark}
For the Kuramoto model with frustration, in Theorem \ref{enter_small}, we show the phase diameter of whole ensemble will be uniformly bounded by a small value $D^\infty$ after some finite time. Under the assumption that $\alpha$ is sufficiently small such that $D^\infty + \alpha < \frac{\pi}{2}$, the interaction function $\cos x$ in the dynamics of frequency is positive after the finite time. Thus, we can lift \eqref{KuM} to the second-order formulation, which enjoys the similar form to Cucker-Smale model with the interaction function $\cos x$.

More precisely, we can introduce phase velocity or frequency $\omega_i(t) := \dot{\theta}_i(t)$ for each oscillator, and 
directly differentiate \eqref{KuM} with respect to time $t$ to derive the equivalent second-order Cucker-Smale type model as below
\begin{equation}\label{s_KuM}
\begin{cases} 
\displaystyle \dot{\theta}_i(t) = \omega_i(t), \quad t > 0, \quad i =1,2,\ldots,N, \\
\displaystyle \dot{\omega}_i(t) = K \sum_{j \in \mathcal{N}_i} \cos (\theta_j(t) - \theta_i(t) + \alpha) (\omega_j(t) - \omega_i(t)), \\
\displaystyle (\theta_i(0),\omega_i(0)) =(\theta_i(0), \dot{\theta}_i(0)).
\end{cases}
\end{equation}
\end{remark}
Now for the second-order system \eqref{s_KuM}, we apply the results in \cite{D-H-K20} for Kuramoto model without frustration on a general digraph and present the frequency synchronization for Kuramoto model with frustrations. 
 \begin{corollary}\label{complete_syn}
Let $\theta_i$ be a solution to system \eqref{s_KuM} and suppose the assumptions in Lemma \ref{L4.2} are fulfilled. Moreover, assume that there exists time $t_* > 0$ such that
\begin{equation}\label{small_region}
D(\theta(t)) \le D^\infty, \quad \ t \in [t_*, +\infty),
\end{equation}
where $D^\infty < \min\{\zeta,\frac{\pi}{2}\}$ is a small positive constant and $\alpha$ is sufficiently small such that $D^\infty + \alpha < \frac{\pi}{2}$. Then there exist positive constants $C_1$ and $C_2$ such that
\begin{equation*}
D(\omega(t)) \le C_1 e^{-C_2 (t- t_*)}, \quad t > t_*,
\end{equation*}
where $D(\omega(t)) = \max_{1\le i \le N} \{\omega_i(t)\} - \min_{1\le i \le N} \{\omega_i(t)\}$ is the diameter of phase velocity.
\end{corollary}
\begin{proof}
We can apply Theorem \ref{enter_small} and the methods and results in the work of Dong et al. \cite{D-H-K20} for Kuramoto model without frustration to yield the emergence of exponentially fast synchronization in \eqref{KuM} and \eqref{s_KuM}. As the proof is almost the same as that in \cite{D-H-K20}, we omit its details.\newline 
\end{proof}

\section{Summary}\label{sec:5}
\setcounter{equation}{0}
In this paper, under the effect of frustration, we provide sufficient frameworks leading to the complete synchronization for the Kuramoto model with general network containing a spanning tree. To this end, we follow a node decomposition introduced in \cite{H-L-Z20} and construct hypo-coercive inequalities through which we can study the upper bounds of phase diameters. When the initial configuration is confined in a half circle, for sufficiently small frustration and sufficiently large coupling strength, we show that the relative differences of Kuramoto oscillators adding a phase shift will be confined into a small region less than a quarter circle in finite time, thus we can directly apply the methods and results in \cite{D-H-K20} to prove that the complete synchronization emerges exponentially fast. 
\vspace{0.5cm}

\begin{appendix}
\setcounter{equation}{0}
\section{proof of Lemma \ref{dynamics_Q0}}\label{appendix1}
We will split the proof into six steps. In the first step, we suppose by contrary that the phase diameter of $\mathcal{G}_0$ is bounded by $\gamma$ in a finite time interval. In the second, third and forth steps, we use induction criteria to construct the differential inequality of $Q^0(t)$ in the finite time interval. In the last two steps, we exploit the derived differential inequality of $Q^0(t)$ to conclude that phase diameter of $\mathcal{G}_0$ is bounded by $\gamma$ on $[0, +\infty)$, and thus the differential inequality of $Q^0(t)$ obtained in the forth step also holds on $[0, +\infty)$.\newline

\noindent $\bigstar$ \textbf{Step 1.} Define a set
\begin{equation*}
\mathcal{B}_0 := \{ T >0 : \ D_0(\theta(t)) < \gamma, \ \forall \ t \in [0,T) \}.
\end{equation*}
From Lemma \ref{diameter_zeta} where $N = N_0$ in the present section, the set $\mathcal{B}_0$ is non-empty since
\[D_0(\theta(t)) = D(\theta(t)) < \zeta < \gamma, \quad \forall \ t \in [0,\bar{t}),\]
which directly yields that $\bar{t} \in \mathcal{B}_0$. Define $T^* = \sup \mathcal{B}_0$. And we claim that $T^* = +\infty$. Suppose not, i.e., $T^* < +\infty$, then we apply the continuity of $D_0(\theta(t))$ to have
\begin{equation}\label{C-0}
D_0(\theta(t)) < \gamma, \quad \forall \ t\in [0,T^*), \quad D_0(\theta(T^*)) = \gamma.
\end{equation}
In particular, we have $\bar{t} \le T^*$. The analyticity of the solution to system \eqref{KuM} is guaranteed by the standard Cauchy-Lipschitz theory. Therefore, in the finite time interval $[0,T^*)$, any two oscillators either collide finite times or always stay together. If there are some $\theta_i$ and $\theta_j$ always staying together in $[0,T^*]$, we can view them as one oscillator and thus the total number of oscillators that we need to study can be reduced. This is a more simpler situation, and we can similarly deal with it. Therefore, we only consider the case that there is no pair of oscillators staying together in $[0,T^*)$. For this case, there are only finite many collisions occurring through $[0,T^*)$. Thus, we divide the time interval $[0,T^*)$ into a finite union as below
\[[0,T^*) = \bigcup_{l=1}^r J_l, \quad J_l = [t_{l-1},t_l),\]
where the end point $t_l$ denotes the collision instant. It is easy to see that there is no collision in the interior of $J_l$. Now we pick out any time interval $J_l$ and assume that
\begin{equation}\label{C-1}
\theta_1^0(t) \le \theta^0_2(t) \le \ldots \le \theta^0_{N_0}(t), \quad t \in J_l. 
\end{equation}
\noindent $\bigstar$ \textbf{Step 2.} According to the notations in \eqref{abbreviation_0}, we follow the process $\mathcal{A}_1$ and $\mathcal{A}_2$ to construct $\bar{\theta}^0_n$ and $\underline{\theta}^0_n, \ 1\le n\le N_0$, respecively.
We first study the dynamics of $\bar{\theta}_{N_0}^0 = \theta_{N_0}^0$,
\begin{equation}\label{C1-1}
\begin{aligned}
\dot{\theta}^0_{N_0}(t) &= \Omega^0_{N_0} + K \sum_{j \in \mathcal{N}^0_{N_0}(0)} \sin (\theta_j^0 - \theta^0_{N_0} + \alpha) \\
& \le \Omega_M + K \sum_{j \in \mathcal{N}^0_{N_0}(0)} \left[\sin (\theta^0_j - \theta^0_{N_0}) \cos \alpha + \cos (\theta^0_j - \theta^0_{N_0}) \sin \alpha\right]\\
& \le \Omega_M + N_0K \sin \alpha + K \cos \alpha \min_{j \in \mathcal{N}^0_{N_0}(0)} \sin (\theta^0_j - \theta^0_{N_0}).
\end{aligned}
\end{equation}
For the dynamics of $\bar{\theta}^0_{N_0-1}$, according to the process $\mathcal{A}_1$ and $\bar{a}^0_{N_0-1} = \eta(N_0 + 2)$ in \eqref{coeffi_0}, we apply \eqref{C1-1} and estimate the derivative of $\bar{\theta}^0_{N_0-1}$ as follows,
\begin{align}
\dot{\bar{\theta}}^0_{N_0-1} &= \frac{d}{dt} \left(\frac{\bar{a}^0_{N_0-1} \theta^0_{N_0} + \theta^0_{N_0-1}}{\bar{a}^0_{N_0-1} +1}\right) = \frac{\bar{a}^0_{N_0-1}}{\bar{a}^0_{N_0-1} + 1}\dot{\theta}^0_{N_0} + \frac{1}{\bar{a}^0_{N_0-1}+1}\dot{\theta}^0_{N_0-1} \label{C-2}\\
& \le \frac{\bar{a}^0_{N_0-1}}{\bar{a}^0_{N_0-1} + 1} \left(\Omega_M + N_0K \sin \alpha + K \cos \alpha \min_{j \in \mathcal{N}^0_{N_0}(0)} \sin (\theta^0_j - \theta^0_{N_0}) \right) \notag\\
&+ \frac{1}{\bar{a}^0_{N_0-1}+1} \left(\Omega^0_{N_0-1} + K \sum_{j \in \mathcal{N}^0_{N_0-1}(0)} \sin (\theta^0_j - \theta^0_{N_0-1} + \alpha)\right)  \notag\\
&\le \Omega_M + K \cos \alpha \frac{1}{\bar{a}^0_{N_0-1} + 1} \eta (N_0 +2) \min_{j \in \mathcal{N}^0_{N_0}(0)} \sin (\theta^0_j - \theta^0_{N_0})   \notag\\
&+ \frac{\bar{a}^0_{N_0-1}}{\bar{a}^0_{N_0-1} + 1} N_0 K \sin \alpha + K\cos \alpha \frac{1}{\bar{a}^0_{N_0-1} +1} \sum_{j \in \mathcal{N}^0_{N_0-1}(0)} \sin (\theta^0_j - \theta^0_{N_0-1}) \notag \\
&+ K \sin \alpha \frac{1}{\bar{a}^0_{N_0-1} + 1} \sum_{j \in \mathcal{N}^0_{N_0 -1}(0) }\cos (\theta^0_j - \theta^0_{N_0 -1}) \notag\\
&\le \Omega_M + K \cos \alpha \frac{1}{\bar{a}^0_{N_0-1} + 1} 2\eta  \min_{j \in \mathcal{N}^0_{N_0}(0)} \sin (\theta^0_j - \theta^0_{N_0}) + \frac{\bar{a}^0_{N_0-1}}{\bar{a}^0_{N_0-1} + 1} N_0 K \sin \alpha   \notag\\
&+ K\cos \alpha \frac{1}{\bar{a}^0_{N_0-1} + 1} \left(\underset{j \le N_0 -1}{\sum_{j \in \mathcal{N}^0_{N_0-1}(0)}} \sin(\theta^0_j - \theta^0_{N_0-1}) + \sin (\theta^0_{N_0} - \theta^0_{N_0 -1})\right) \notag\\
&+\frac{1}{\bar{a}^0_{N_0-1} + 1} N_0 K \sin \alpha  \notag\\
&\le \Omega_M + K \cos \alpha \frac{1}{\bar{a}^0_{N_0-1} + 1} \eta  \min_{j \in \mathcal{N}^0_{N_0}(0)} \sin (\theta^0_j - \theta^0_{N_0}) \notag \\
&+K\cos \alpha \frac{1}{\bar{a}^0_{N_0-1} + 1}  \underset{j \le N_0 -1}{\min_{j \in \mathcal{N}^0_{N_0-1}(0)}} \sin (\theta^0_j - \theta^0_{N_0-1})  \notag\\
&+ K\cos \alpha \frac{1}{\bar{a}^0_{N_0-1} + 1}  \underbrace{\left(\eta  \min_{j \in \mathcal{N}^0_{N_0}(0)} \sin (\theta^0_j - \theta^0_{N_0}) + \sin (\theta^0_{N_0} - \theta^0_{N_0 -1})\right)}_{\mathcal{I}_1} + N_0K\sin \alpha, \notag
\end{align}
where we use
\begin{equation*}
\begin{aligned}
&|\sum_{j \in \mathcal{N}^0_{N_0-1}(0)} \cos (\theta^0_j - \theta^0_{N_0-1})| \le N_0, \quad K \cos \alpha \frac{1}{\bar{a}^0_{N_0-1} + 1} \eta N_0 \min_{j \in \mathcal{N}^0_{N_0}(0)} \sin (\theta^0_j - \theta^0_{N_0}) \le 0,\\
&\underset{j \le N_0 -1}{\sum_{j \in \mathcal{N}^0_{N_0-1}(0)}} \sin(\theta^0_j - \theta^0_{N_0-1})  \le \underset{j \le N_0 -1}{\min_{j \in \mathcal{N}^0_{N_0-1}(0)}} \sin (\theta^0_j - \theta^0_{N_0-1}). 
\end{aligned}
\end{equation*}
Next we show the term $\mathcal{I}_1$ is non-positive.  We only consider the situation $\gamma > \frac{\pi}{2}$, and the case $\gamma \le \frac{\pi}{2}$ can be similarly dealt with. It is clear that
\[\min_{j \in \mathcal{N}^0_{N_0}(0)} \sin(\theta^0_j - \theta^0_{N_0})  \le \sin(\theta^0_{\bar{k}_{N_0}} - \theta^0_{N_0}) \quad \text{where} \  \bar{k}_{N_0} = \min_{j \in  \mathcal{N}^0_{N_0}(0)} j. \]
Note that $\bar{k}_{N_0} \le N_0-1$ since $\bar{\mathcal{L}}^{N_0}_{N_0}(\bar{C}_{N_0,N_0})$ is not a general root. Therefore, if $ 0 \le \theta^0_{N_0}(t) - \theta^0_{\bar{k}_{N_0}}(t) \le \frac{\pi}{2}$, we immediately obtain that
\begin{equation*}
0 \le \theta^0_{N_0}(t) - \theta^0_{N_0-1}(t) \le \theta^0_{N_0}(t) - \theta^0_{\bar{k}_{N_0}}(t) \le \frac{\pi}{2}, 
\end{equation*}
which implies that
\begin{equation*}
\mathcal{I}_1 \le \eta \sin(\theta^0_{\bar{k}_{N_0}} - \theta^0_{N_0}) + \sin(\theta^0_{N_0} - \theta^0_{N_0-1}) \le \sin(\theta^0_{\bar{k}_{N_0}} - \theta^0_{N_0}) + \sin(\theta^0_{N_0} - \theta^0_{N_0-1}) \le 0.
\end{equation*}
On the other hand, if $\frac{\pi}{2} <  \theta^0_{N_0}(t) - \theta^0_{\bar{k}_{N_0}}(t) < \gamma$, we use the fact
\[\eta > \frac{1}{\sin \gamma} \quad \mbox{and} \quad \sin(\theta^0_{N_0}(t) - \theta^0_{\bar{k}_{N_0}}(t)) > \sin \gamma,\] 
to conclude that $\eta \sin(\theta^0_{\bar{k}_{N_0}} - \theta^0_{N_0}) \le -1$. Hence, in this case, we still obtain that
\begin{equation*}
\mathcal{I}_1 \le \eta \sin(\theta^0_{\bar{k}_{N_0}} - \theta^0_{N_0}) + \sin(\theta^0_{N_0} - \theta^0_{N_0-1}) \le -1 + 1 \le 0.
\end{equation*}
Thus, for $t\in J_l$, we combine above analysis to conclude that 
\begin{equation}\label{C-3}
\mathcal{I}_1 = \eta \min_{j \in \mathcal{N}^0_{N_0}(0)} \sin(\theta^0_j - \theta^0_{N_0}) + \sin(\theta^0_{N_0} - \theta^0_{N_0-1}) \le 0.
\end{equation}
Then combining \eqref{C-2} and \eqref{C-3}, we derive that
\begin{equation}\label{C-4}
\begin{aligned}
\dot{\bar{\theta}}^0_{N_0-1} &\le  \Omega_M + N_0 K \sin \alpha \\
& + K \cos \alpha \frac{1}{\bar{a}^0_{N_0-1} + 1} \left(\eta  \min_{j \in \mathcal{N}^0_{N_0}(0)} \sin (\theta^0_j - \theta^0_{N_0}) + \underset{j \le N_0 -1}{\min_{j \in \mathcal{N}^0_{N_0-1}(0)}} \sin (\theta^0_j - \theta^0_{N_0-1}) \right). 
\end{aligned}
\end{equation}
$\ $

\noindent $\bigstar$ \textbf{Step 3.} Now we apply the induction principle to cope with $\bar{\theta}^0_n$ in \eqref{abbreviation_0},  which is construced in the iteration process $\mathcal{A}_1$. We will prove for $1 \le n \le N_0$ that,
\begin{equation}\label{C-5}
\dot{\bar{\theta}}^0_n(t) \le \Omega_M + N_0K\sin \alpha + K \cos \alpha \frac{1}{\bar{a}^0_{n} + 1} \sum_{i=n}^{N_0} \eta^{i-n} \underset{j \le i}{\min_{j \in \mathcal{N}^0_{i}(0)}} \sin (\theta^0_j(t) - \theta^0_{i}(t)) 
\end{equation}
In fact, it is known that \eqref{C-5} already holds for $n = N_0, N_0-1$ from \eqref{C1-1} and \eqref{C-4}. Then by induction criteria, suppose \eqref{C-5} holds for $n$, 
Next we verify that \eqref{C-5} still holds for $n-1$. According to the Algorithm $\mathcal{A}_1$ and similar calculations in \eqref{C-2}, the dynamics of the quantity $\bar{\theta}^0_{n-1}(t)$ subjects to the following estimates
\begin{equation}\label{C-7}
\begin{aligned}
\dot{\bar{\theta}}^0_{n-1} 
&\le \Omega_M + N_0K \sin \alpha\\
&+ K \cos \alpha \frac{1}{\bar{a}^0_{n-1} + 1} \eta\sum_{i=n}^{N_0} \eta^{i-n} \underset{j \le i}{\min_{j \in \mathcal{N}^0_{i}(0)}} \sin (\theta^0_j - \theta^0_{i}) \\
&+ K \cos \alpha \frac{1}{\bar{a}^0_{n-1}+1} \underset{j \le n-1}{\min_{j\in \mathcal{N}_{n-1}^0(0)}} \sin (\theta^0_j - \theta^0_{n-1})\\
&+K \cos \alpha \frac{1}{\bar{a}^0_{n-1} + 1} \\
&\times\left(\underbrace{\eta(N_0 - n +1)\sum_{i=n}^{N_0} \eta^{i-n} \underset{j \le i}{\min_{j \in \mathcal{N}^0_{i}(0)}} \sin (\theta^0_j - \theta^0_{i}) + \underset{j > n-1}{\sum_{j \in \mathcal{N}^0_{n-1}(0)}} \sin (\theta^0_j - \theta^0_{n-1})}_{\mathcal{I}_2}\right) ,
\end{aligned}
\end{equation}
Moreover, we can prove the term $\mathcal{I}_2$ is non-positive. As the proof is very similar as that in the previous step, we omit the details and directly claim that $\mathcal{I}_2 \le 0$, which together with \eqref{C-7} verifies \eqref{C-5}.\newline

\noindent $\bigstar$ \textbf{Step 4.}
Now, we set $n=1$ in \eqref{C-5} and apply Lemma \ref{eta_sin_inequality} to have
\begin{equation}\label{C-10}
\begin{aligned}
\dot{\bar{\theta}}^0_{1} &\le \Omega_M + N_0 K \sin \alpha + K \cos \alpha \frac{1}{\bar{a}^0_1+1} \sum_{i=1}^{N_0} \eta^{i-1} \underset{j \le i}{\min_{j \in \mathcal{N}^0_{i}(0)}} \sin (\theta^0_j - \theta^0_{i})  \\
&\le \Omega_M + N_0K \sin \alpha+ K \cos \alpha \frac{1}{\bar{a}^0_1+1} \sin (\theta^0_{\bar{k}_1} - \theta^0_{N_0}) \\
&= \Omega_M + N_0K \sin \alpha+ K \cos \alpha \frac{1}{\bar{a}^0_1+1} \sin (\theta^0_{1} - \theta^0_{N_0}) ,
\end{aligned}
\end{equation}
where $\bar{k}_{1} = \min_{j \in \bigcup_{i=1}^{N_0} \mathcal{N}^0_{i}(0)} j =1$ due to the strong connectivity of $\mathcal{G}_0$. Similarly, we can follow the process $\mathcal{A}_2$ to construct $\underline{\theta}^0_k$ in \eqref{abbreviation_0} until $k = N_0$. Then, we can apply the similar argument in \eqref{C-5} to obtain that,
\begin{equation}\label{C-11}
\begin{aligned}
\frac{d}{dt} \underline{\theta}^0_{N_0} (t) 
& \ge \Omega_m - N_0K \sin \alpha + K \cos \alpha \frac{1}{\bar{a}^0_1 + 1} \sin (\theta^0_{N_0} - \theta^0_1) ,
\end{aligned}
\end{equation}
Then we recall the notations $\bar{\theta}_0 = \bar{\theta}^0_1$ and $\underline{\theta}_0 = \underline{\theta}^0_{N_0}$, and combine \eqref{C-10} and \eqref{C-11} to obtain that
\begin{equation*}
\begin{aligned}
\dot{Q}^0(t) &= \frac{d}{dt}(\bar{\theta}_0 - \underline{\theta}_0) \le D(\Omega) + 2N_0K \sin \alpha- K \cos \alpha \frac{2}{\bar{a}^0_1 + 1} \sin (\theta^0_{N_0} - \theta^0_1)  \\
&\le D(\Omega) + 2N_0K \sin \alpha - K\cos \alpha \frac{1}{\sum_{j=1}^{N_0-1} \eta^j A(2N_0,j) +1} \sin (\theta^0_{N_0} - \theta^0_1) 
\end{aligned}
\end{equation*}
where we use the property
\[\bar{a}^0_{1} = \sum_{j=1}^{N_0-1} \eta^j A(2N_0,j) .\]
As the function  $\frac{\sin x}{x}$ is monotonically decreasing in $(0, \pi]$, we apply \eqref{C-0} to obtain  that
\[\sin (\theta^0_{N_0} - \theta^0_1) \ge \frac{\sin \gamma}{\gamma}(\theta^0_{N_0} - \theta^0_1).\]
Moreover, due to the fact $Q^0(t) \le \theta^0_{N_0}(t) - \theta^0_1(t)$, we have
\begin{equation}\label{C-12}
\begin{aligned}
\dot{Q}^0(t) &\le D(\Omega) + 2N_0 K \sin \alpha - K\cos \alpha \frac{1}{\sum_{j=1}^{N_0-1} \eta^j A(2N_0,j) +1} \frac{\sin \gamma}{\gamma}(\theta^0_{N_0} - \theta^0_1)  \\
&\le D(\Omega) + 2N_0 K \sin \alpha - K\cos \alpha \frac{1}{\sum_{j=1}^{N_0-1} \eta^j A(2N_0,j) +1} \frac{\sin \gamma}{\gamma} Q^0(t) , \quad t \in J_l.
\end{aligned}
\end{equation}
Note that the constructed quantity $Q^0(t) = \bar{\theta}_0(t) - \underline{\theta}_0(t)$ is Lipschitz continuous on $[0,T^*)$. 
Moreover, the above analysis does not depend on the time interval $J_l, \ l=1,2,\ldots,r$, thus the differential inequality \eqref{C-12} holds almost everywhere on $[0,T^*)$. \newline

\noindent $\bigstar$ \textbf{Step 5.} Next we study the upper bound of $Q^0(t)$ in the period $[0,T^*)$. Define
\begin{equation*}
M_0 = \max\left\{Q^0(0), \beta D^\infty\right\}.
\end{equation*}
We claim that
\begin{equation}\label{C-14}
Q^0(t) \le M_0 \quad \mbox{for all} \ t\in [0,T^*).
\end{equation}
Suppose not, then there exists some $\tilde{t} \in [0,T^*)$ such that $Q^0(\tilde{t}) > M_0$. We construct a set
\[\mathcal{C}_0  := \{t < \tilde{t} \ | \ Q^0(t) \le M_0\} .\]
Since $0 \in \mathcal{C}_0$, the set $\mathcal{C}_0$ is not empty. Then we denote $t^* = \sup \mathcal{C}_0$. It is easy to see that 
\begin{equation}\label{C-15}
t^* < \tilde{t}, \quad Q^0(t^*) = M_0, \quad Q^0(t) > M_0 \quad \mbox{for} \ t \in (t^*, \tilde{t}].
\end{equation}
 For a given sufficiently small $D^\infty < \min\{\frac{\pi}{2},\zeta\}$, based on the assumptions about the frustration and the coupling strength in \eqref{condition_3}, it is clear that
\begin{equation}\label{C-13}
K > \left(1+ \frac{\zeta}{\zeta - D(\theta(0)}\right) \frac{(D(\Omega) +2N_0 K\sin \alpha)c}{\cos \alpha} \frac{1}{\beta D^\infty} > \frac{(D(\Omega) +2N_0 K\sin \alpha)c}{\cos \alpha} \frac{1}{\beta D^\infty}
\end{equation}
where
\[c = \frac{\left(\sum_{j=1}^{N_0-1} \eta^j A(2N_0,j) + 1\right)\gamma}{\sin \gamma}.\]
Thus combing the construction of $M_0$, \eqref{C-15} and \eqref{C-13}, we obtain that for $t \in (t^*, \tilde{t}]$, the following estimate holds ,
\begin{equation*}
\begin{aligned}
&D(\Omega) + 2N_0 K \sin \alpha - K\cos \alpha \frac{1}{\sum_{j=1}^{N_0-1} \eta^j A(2N_0,j) +1} \frac{\sin \gamma}{\gamma} Q^0(t)  \\
& < D(\Omega) + 2N_0 K \sin \alpha - K\cos \alpha \frac{1}{\sum_{j=1}^{N_0-1} \eta^j A(2N_0,j) +1} \frac{\sin \gamma}{\gamma} \beta D^\infty  < 0.
\end{aligned}
\end{equation*}
Then, we apply the above inequality and integrate on the both sides of \eqref{C-12} from $t^*$ to $\tilde{t}$ to get
\begin{equation*}
\begin{aligned}
&Q^0(\tilde{t}) - M_0\\
 &= Q^0(\tilde{t}) - Q^0(t^*) \\
&\le \int_{t^*}^{\tilde{t}} \left(D(\Omega) + 2N_0 K \sin \alpha - K\cos \alpha \frac{1}{\sum_{j=1}^{N_0-1} \eta^j A(2N_0,j) +1} \frac{\sin \gamma}{\gamma} Q^0(t) \right)dt <0,
\end{aligned}
\end{equation*}
which obviously contradicts to the fact $Q^0(\tilde{t}) - M_0 >0$, and verifies \eqref{C-14}.\newline

\noindent $\bigstar$ \textbf{Step 6.} Now we are ready to show the contradiction to \eqref{C-0}, which implies that $T^*=+\infty$. In fact, from the fact that $\beta < 1, D^\infty < \zeta$ and $Q^0(0) \le D_0(\theta(0)) < \zeta$, we see
\begin{equation*}
Q^0(t) \le M_0 = \max\left\{Q^0(0), \beta D^\infty\right\} < \zeta, \quad t\in [0,T^*).
\end{equation*}
Then we apply the relation $\beta D_0(\theta(t)) \le Q^0(t)$ given in Lemma \ref{beta_QD} and the assumption $\eta > \frac{2}{1 - \frac{\zeta}{\gamma}}$ in \eqref{condition_3} to obtain that
\[D_0(\theta(t)) \le \frac{Q^0(t)}{\beta} < \frac{\zeta}{\beta} < \gamma, \quad t \in [0,T^*) \quad \mbox{where} \ \beta = 1 - \frac{2}{\eta}.\]
As $D_0(\theta(t))$ is continuous, we have
\begin{equation*}
D_0(\theta(T^*)) = \lim_{t \to (T^*)^-}D_0(\theta(t)) \le \frac{\zeta}{\beta} < \gamma,
\end{equation*}
which contradicts to the situation that $D_0(\theta(T^*)) = \gamma$ in \eqref{C-0}.
Therefore, we conclude that $T^* = +\infty$, which implies that
\begin{equation}\label{C-a15}
D_0(\theta(t)) < \gamma, \quad \mbox{for all} \ t \in [0, +\infty).
\end{equation}
Then for any finite time $T>0$, we apply \eqref{C-a15} and repeat the same argument in the second, third, forth steps to obtain the dynamics of $Q^0(t)$ in \eqref{C-12} holds on $[0,T)$. Thus we obtain the following differential inequality of $Q^0$ on the whole time interval:
\begin{equation*}
\dot{Q}^0(t) \le D(\Omega) + 2N_0 K \sin \alpha- K\cos \alpha \frac{1}{\sum_{j=1}^{N_0-1} \eta^j A(2N_0,j) +1} \frac{\sin \gamma}{\gamma} Q^0(t) , \ t \in [0, +\infty).
\end{equation*}
\qed

\section{proof of step 1 in lemma \ref{L4.3}}\label{appendix2}
\setcounter{equation}{0}
We will show the detailed proof of Step 1 in Lemma \ref{L4.3}. Now we pick out any interval $J_l$ with $1\le l \le r$, where the orders of both $\{\bar{\theta}_i\}_{i=0}^{k+1}$ and $\{\underline{\theta}_i\}_{i=0}^{k+1}$ are preseved and the order of oscillators in each subdigraph $\mathcal{G}_i$ with $0\le i \le k+1$ will not change in each time interval. Then, we consider four cases depending on the possibility of relative position between $\bigcup_{i=0}^k \mathcal{G}$ and $\mathcal{G}_{k+1}$. 

\begin{figure}[H]
	\centering  
	\subfigbottomskip=20pt 
	\subfigcapskip=-2pt 
	\subfigure[Case 1]{
		\includegraphics[width=0.40\linewidth]{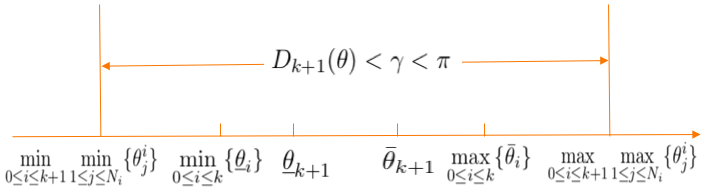}}
		\qquad 
	\subfigure[Case 2]{
		\includegraphics[width=0.40\linewidth]{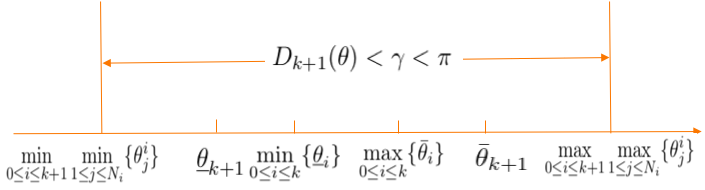}}
	  \\
	\subfigure[Case 3]{
		\includegraphics[width=0.40\linewidth]{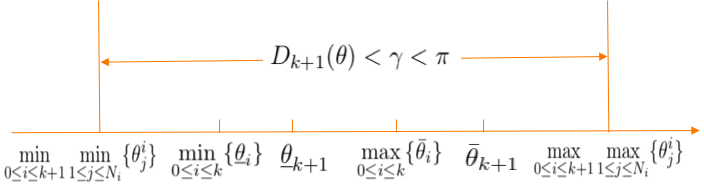}}
	\qquad
	\subfigure[Case 4]{
		\includegraphics[width=0.40\linewidth]{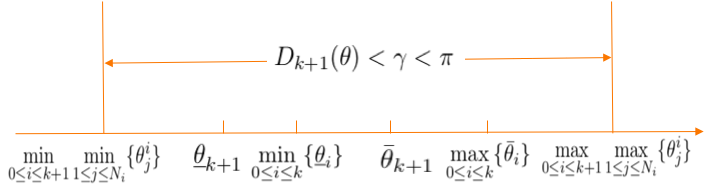}}
	\caption{The four cases}
	\label{Fig1}
\end{figure}
Figure \ref{Fig1} above shows the four possible relations between $\bigcup_{i=0}^k \mathcal{G}$ and $\mathcal{G}_{k+1}$ at any time $t$. Case $1$ and Case $4$ are similar and relative simple, while the analysis on Case $2$ and Case $3$ are similar but much more complicated. Therefore, we will only show the detailed proof of Case $2$ for simplicity. 
In this case, we have from Figure \ref{Fig1} that
\begin{equation*}
\max_{0 \le i\le k+1}\{\bar{\theta}_i\}  = \bar{\theta}_{k+1} , \quad \min_{0\le i\le k+1}\{\underline{\theta}_i\} =  \underline{\theta}_{k+1} \quad \mbox{for $t\in J_l$}.
\end{equation*}
Without loss of generality, we assume that
\begin{equation*}
\theta^{k+1}_1 \le \theta^{k+1}_2 \le \dots \le \theta^{k+1}_{N_{k+1}}, \quad \mbox{for $t\in J_l$}.
\end{equation*}

\noindent $\bigstar$ \textbf{Step 1.} Similar to \eqref{C-5}, we claim that for $1 \le n \le N_{k+1}$, the following inequalities hold
\begin{equation}\label{F-5}
\begin{aligned}
\frac{d}{dt} \bar{\theta}_n^{k+1}(t) &\le \Omega_M + S_{k+1}K\sin \alpha + S_kK \cos \alpha D_k (\theta(t))\\
&+ K \cos \alpha \frac{1}{\bar{a}^{k+1}_n + 1}\sum_{i=n}^{N_{k+1}}\left( \eta^{i-n} \underset{j \le i}{\min_{j\in \mathcal{N}_{i}^{k+1}(k+1)}} \sin (\theta^{k+1}_j(t) - \theta^{k+1}_{i}(t))\right).
\end{aligned}
\end{equation}
where $S_k = \sum_{i=0}^k N_i$. In the following, we will prove the claim \eqref{F-5} via induction principle.\newline

\noindent $\bigstar$ \textbf{Step 1.1.} As an initial step, we first verify that \eqref{F-5} holds for $n = N_{k+1}$. In fact, we have
\begin{equation}\label{F-d6}
\begin{aligned}
\frac{d}{dt} \bar{\theta}^{k+1}_{N_{k+1}} &= \frac{d}{dt} \theta^{k+1}_{N_{k+1}}\\
&= \Omega^{k+1}_{N_{k+1}} + K \cos \alpha \sum_{j \in \mathcal{N}^{k+1}_{N_{k+1}}(k+1)} \sin (\theta^{k+1}_j - \theta^{k+1}_{N_{k+1}}) \\
&+ K \cos \alpha \sum_{l=0}^k \sum_{j \in \mathcal{N}^{k+1}_{N_{k+1}}(l)} \sin (\theta^l_j - \theta^{k+1}_{N_{k+1}}) +K \sin \alpha \sum_{l=0}^{k+1} \sum_{j \in \mathcal{N}^{k+1}_{N_{k+1}}(l)} \cos (\theta^l_j - \theta^{k+1}_{N_{k+1}}) \\
&\le \Omega_M + S_{k+1}K \sin \alpha \\
&+K \cos \alpha \underbrace{\sum_{j \in \mathcal{N}^{k+1}_{N_{k+1}}(k+1)} \sin (\theta^{k+1}_j - \theta^{k+1}_{N_{k+1}})}_{\mathcal{I}_{11}} + K \cos \alpha \underbrace{\sum_{l=0}^k \sum_{j \in \mathcal{N}^{k+1}_{N_{k+1}}(l)} \sin (\theta^l_j - \theta^{k+1}_{N_{k+1}})}_{\mathcal{I}_{12}},
\end{aligned}
\end{equation}
where we use
\[|\sum_{l=0}^{k+1} \sum_{j \in \mathcal{N}^{k+1}_{N_{k+1}}(l)} \cos (\theta^l_j - \theta^{k+1}_{N_{k+1}})| \le \sum_{l=0}^{k+1} N_l = S_{k+1}.\]

\noindent $\diamond$ \textbf{Estimates on $\mathbf{\mathcal{I}_{11}}$ in \eqref{F-d6}}. We know that $\theta^{k+1}  _{N_{k+1}}$ is the largest phase among $\mathcal{G}_{k+1}$, and all the oscillators in $\bigcup_{i=0}^{k+1} \mathcal{G}_{i}$ are confined in half circle before $T^*$. Therefore, it is clear that
\begin{equation*}
\sin (\theta^{k+1}_j - \theta^{k+1}_{N_{k+1}}) \le 0, \quad \mbox{for} \ j \in \mathcal{N}_{N_{k+1}}^{k+1}(k+1).
\end{equation*}
Then we immediately have 
\begin{equation}\label{F-6}
\mathcal{I}_{11} = \sum_{j\in \mathcal{N}_{N_{k+1}}^{k+1}(k+1)} \sin (\theta^{k+1}_j - \theta^{k+1}_{N_{k+1}}) \le \min_{j \in \mathcal{N}^{k+1}_{N_{k+1}}(k+1)} \sin (\theta^{k+1}_j - \theta^{k+1}_{N_{k+1}}).
\end{equation}

\noindent $\diamond$ \textbf{Estimates on $\mathbf{\mathcal{I}_{12}}$ in \eqref{F-d6}}. For $\theta^l_j$ which is the neighbor of $\theta^{k+1}_{N_{k+1}}$ in $\mathcal{G}_l$ with $0 \le l \le k$, i.e., $j\in \mathcal{N}_{N_{k+1}}^{k+1}(l)$, we consider two possible orderings between  $\theta^l_j$ and $\theta^{k+1}_{N_{k+1}}$:

\noindent If $\theta^l_j \le \theta^{k+1}_{N_{k+1}}$, we immediately have
\begin{equation*}
\sin (\theta^{l}_j - \theta^{k+1}_{N_{k+1}}) \le 0.
\end{equation*}
\noindent If $\theta^l_j > \theta^{k+1}_{N_{k+1}}$, from the fact that
\begin{equation}\label{F-d7}
 \theta^{i}_{N_i} \ge \bar{\theta}_i \ge \underline{\theta}_i \ge \theta^i_1, \quad 0 \le i \le d,
 \end{equation}
we immediately obtain 
\begin{equation}\label{F-7}
\theta^{k+1}_{N_{k+1}} \ge \bar{\theta}_{k+1}  = \max_{0 \le i\le k+1}\{\bar{\theta}_i\}  \ge \max_{0 \le i\le k}\{\bar{\theta}_i\} \ge \min_{0 \le i\le k}\{\underline{\theta}_i\} \ge \min_{0 \le i \le k}\min_{1 \le j \le N_i} \{\theta^i_j\}.
\end{equation}
Thus we use the property of $\sin x  \le x, \ x \ge 0$ and \eqref{F-7} to get
\begin{equation*}\label{F-d8}
\sin (\theta^{l}_j - \theta^{k+1}_{N_{k+1}}) \le \theta^{l}_j - \theta^{k+1}_{N_{k+1}} \le \theta^{l}_j -  \min_{0 \le i \le k}\min_{1 \le j \le N_i} \{\theta^i_j\} \le D_k(\theta(t)).
\end{equation*}
Therefore, combining the above discussion, we have
\begin{equation}\label{F-d8}
\mathcal{I}_{12} = \sum_{l=0}^k \sum_{j \in \mathcal{N}^{k+1}_{N_{k+1}}(l)} \sin (\theta^l_j - \theta^{k+1}_{N_{k+1}}) \le S_k D_k(\theta(t))
\end{equation}
From \eqref{F-d6}, \eqref{F-6} and \eqref{F-d8}, it yields that
 \eqref{F-5} holds for $n = N_{k+1}$.\newline

\noindent $\bigstar$ \textbf{Step 1.2.} Next, we assume that \eqref{F-5} holds for $n$ with $2 \le n \le N_{k+1}$, and we  will show that \eqref{F-5} holds for $n-1$. Following the process $\mathcal{A}_1$ and similar analysis in \eqref{F-d6}, we have
\begin{equation}\label{F-8}
\begin{aligned}
\dot{\bar{\theta}}^{k+1}_{n-1} &\le \Omega_M + \frac{\bar{a}^{k+1}_{n-1}}{\bar{a}^{k+1}_{n-1}+1} S_{k+1}K\sin \alpha + \frac{\bar{a}^{k+1}_{n-1}}{\bar{a}^{k+1}_{n-1}+1} S_kK \cos \alpha D_k (\theta(t))+ K \sin\alpha \frac{1}{\bar{a}^{k+1}_{n-1}+1}S_{k+1} \\
&+ K\cos \alpha \frac{1}{\bar{a}^{k+1}_{n-1} + 1} \eta (N_{k+1}-n+2+S_k) \sum_{i=n}^{N_{k+1}}\left( \eta^{i-n} \underset{j \le i}{\min_{j\in \mathcal{N}_{i}^{k+1}(k+1)}} \sin (\theta^{k+1}_j - \theta^{k+1}_{i})\right) \\
&+ K \cos\alpha \frac{1}{\bar{a}^{k+1}_{n-1}+1} \underset{j \le n-1}{\min_{j\in \mathcal{N}_{n-1}^{k+1}(k+1)}} \sin (\theta^{k+1}_j - \theta^{k+1}_{n-1}) \\
&+ K \cos\alpha \frac{1}{\bar{a}^{k+1}_{n-1}+1}\underbrace{\underset{j > n-1}{\sum_{j\in \mathcal{N}_{n-1}^{k+1}(k+1)}}\sin (\theta^{k+1}_j - \theta^{k+1}_{n-1})}_{\mathcal{I}_{21}}\\
&+ K \cos\alpha \frac{1}{\bar{a}^{k+1}_{n-1}+1} \underbrace{\sum_{l=0}^k \sum_{j\in \mathcal{N}_{n-1}^{k+1}(l)} \sin (\theta^{l}_j - \theta^{k+1}_{n-1})}_{\mathcal{I}_{22}}.
\end{aligned}
\end{equation}
Next we do some estimates about the terms $\mathcal{I}_{21}$ and $\mathcal{I}_{22}$ in \eqref{F-8} seperately.\newline

\noindent $\diamond$ \textbf{Estimates on $\mathbf{\mathcal{I}_{21}}$ in \eqref{F-8}.} Without loss of generality, we only deal with $\mathcal{I}_{21}$ under the situation $\gamma > \frac{\pi}{2}$. We first apply the strong connectivity of $\mathcal{G}_{k+1}$ and Lemma \ref{eta_sin_inequality} to obtain that
\begin{equation}\label{F-9}
\sum_{i=n}^{N_{k+1}}\left( \eta^{i-n} \underset{j \le i}{\min_{j\in \mathcal{N}_{i}^{k+1}(k+1)}} \sin (\theta^{k+1}_j(t) - \theta^{k+1}_{i}(t))\right) \le \sin(\theta^{k+1}_{\bar{k}_{n}} - \theta^{k+1}_{N_{k+1}}), 
\end{equation}
where $\bar{k}_{n} = \min_{j \in \bigcup_{i=n}^{N_{k+1}} \mathcal{N}^{k+1}_{i}(k+1)} j \le n-1$.
According to \eqref{F-9}, we consider two cases depending on comparison between $\theta^{k+1}_{N_{k+1}} - \theta^{k+1}_{\bar{k}_{n}}$ and $\frac{\pi}{2}$. \newline
\noindent \textbf{(i)} For the first case that $ 0 \le \theta^{k+1}_{N_{k+1}} - \theta^{k+1}_{\bar{k}_{n}} \le \frac{\pi}{2}$, we immediately obtain that for $j \in \mathcal{N}^{k+1}_{n-1}(k+1), \ j>n-1$,
\begin{equation}\label{F-10}
0 \le \theta^{k+1}_{j}(t) - \theta^{k+1}_{n-1}(t) \le \theta^{k+1}_{N_{k+1}}(t) - \theta^{k+1}_{n-1}(t) \le \theta^{k+1}_{N_{k+1}}(t) - \theta^{k+1}_{\bar{k}_{n}}(t) \le \frac{\pi}{2}.
\end{equation}
Then it yieldst from \eqref{F-9}, \eqref{F-10} and $\eta > 2$ that
\begin{equation*}
\begin{aligned}
&\eta(N_{k+1} -  n + 1) \sum_{i=n}^{N_{k+1}}\left( \eta^{i-n} \underset{j \le i}{\min_{j\in \mathcal{N}_{i}^{k+1}(k+1)}} \sin (\theta^{k+1}_j(t) - \theta^{k+1}_{i}(t))\right) + \mathcal{I}_{21} \\
& \le \eta(N_{k+1} -  n + 1) \sin(\theta^{k+1}_{\bar{k}_{n}} - \theta^{k+1}_{N_{k+1}}) + \underset{j > n-1}{\sum_{j\in \mathcal{N}_{n-1}^{k+1}(k+1)}}\sin (\theta^{k+1}_j - \theta^{k+1}_{n-1}) \\
& \le (N_{k+1} -  n + 1) \sin(\theta^{k+1}_{\bar{k}_{n}} - \theta^{k+1}_{N_{k+1}}) + (N_{k+1} -  n + 1) \sin (\theta^{k+1}_{N_{k+1}} - \theta^{k+1}_{n-1}) \\
&\le 0.
\end{aligned}
\end{equation*}
\noindent \textbf{(ii)} For the second case that $ \frac{\pi}{2} < \theta^{k+1}_{N_{k+1}} - \theta^{k+1}_{\bar{k}_{n}} < \gamma$, it is known that
\begin{equation}\label{F-11}
\eta > \frac{1}{\sin \gamma} \quad \mbox{and} \quad \sin(\theta^{k+1}_{N_{k+1}} - \theta^{k+1}_{\bar{k}_{n}}) > \sin \gamma,
\end{equation}
which yields $\eta \sin(\theta^{k+1}_{\bar{k}_{n}} - \theta^{k+1}_{N_{k+1}} ) \le -1$. Thus we immediately derive that
\begin{equation*}
\begin{aligned}
&\eta(N_{k+1} -  n + 1) \sum_{i=n}^{N_{k+1}}\left( \eta^{i-n} \underset{j \le i}{\min_{j\in \mathcal{N}_{i}^{k+1}(k+1)}} \sin (\theta^{k+1}_j(t) - \theta^{k+1}_{i}(t))\right) + \mathcal{I}_{21} \\
& \le \eta(N_{k+1} -  n + 1) \sin(\theta^{k+1}_{\bar{k}_{n}} - \theta^{k+1}_{N_{k+1}}) + \underset{j > n-1}{\sum_{j\in \mathcal{N}_{n-1}^{k+1}(k+1)}}\sin (\theta^{k+1}_j - \theta^{k+1}_{n-1}) \\
& \le -(N_{k+1} -  n + 1) +(N_{k+1} -  n + 1) = 0.
\end{aligned}
\end{equation*}

Then, we combine the above arguments in (i) and (ii) to obtain
\begin{equation}\label{F-12}
\eta(N_{k+1} -  n + 1) \sum_{i=n}^{N_{k+1}}\left( \eta^{i-n} \underset{j \le i}{\min_{j\in \mathcal{N}_{i}^{k+1}(k+1)}} \sin (\theta^{k+1}_j - \theta^{k+1}_{i})\right) + \mathcal{I}_{21} \le 0.
\end{equation}

\noindent $\diamond$ \textbf{Estimates on $\mathbf{\mathcal{I}_{22}}$ in \eqref{F-8}.} For the term $\mathcal{I}_{22}$, there are three possible relations between $\theta^{k+1}_{n-1}$ and  $\theta^l_j$ with $0 \le l \le k$:\newline

\noindent \textbf{(i)} If $\theta^l_j \le \theta^{k+1}_{n-1}$, we immediately have $\sin (\theta^{l}_j - \theta^{k+1}_{n-1}) \le 0$.\newline

\noindent \textbf{(ii)} If $ \theta^{k+1}_{n-1} < \theta^l_j \le \theta^{k+1}_{N_{k+1}}$, we consider two cases separately:\newline

\noindent (a) For the case that $ 0 \le \theta^{k+1}_{N_{k+1}} - \theta^{k+1}_{\bar{k}_{n}} \le \frac{\pi}{2}$, it is clear that
\[0 \le \theta^l_j - \theta^{k+1}_{n-1} \le \theta^{k+1}_{N_{k+1}} - \theta^{k+1}_{n-1}\le \theta^{k+1}_{N_{k+1}} - \theta^{k+1}_{\bar{k}_{n}} \le \frac{\pi}{2}.\]
Thus from the above inequality and \eqref{F-9}, we have
\begin{equation*}
\begin{aligned}
&\eta \sum_{i=n}^{N_{k+1}}\left( \eta^{i-n} \underset{j \le i}{\min_{j\in \mathcal{N}_{i}^{k+1}(k+1)}} \sin (\theta^{k+1}_j(t) - \theta^{k+1}_{i}(t))\right) + \sin (\theta^l_j - \theta^{k+1}_{n-1}) \\
&\le \eta \sin(\theta^{k+1}_{\bar{k}_{n}} - \theta^{k+1}_{N_{k+1}}) + \sin (\theta^l_j - \theta^{k+1}_{n-1}) \\
&\le \sin(\theta^{k+1}_{\bar{k}_{n}} - \theta^{k+1}_{N_{k+1}}) + \sin(\theta^{k+1}_{N_{k+1}} - \theta^{k+1}_{\bar{k}_{n}}) = 0.
\end{aligned} 
\end{equation*}
\noindent (b) For another case that $ \frac{\pi}{2} < \theta^{k+1}_{N_{k+1}} - \theta^{k+1}_{\bar{k}_{n}} < \gamma$, it is yields from \eqref{F-11} that
\begin{equation*}
\begin{aligned}
&\eta \sum_{i=n}^{N_{k+1}}\left( \eta^{i-n} \underset{j \le i}{\min_{j\in \mathcal{N}_{i}^{k+1}(k+1)}} \sin (\theta^{k+1}_j(t) - \theta^{k+1}_{i}(t))\right) + \sin (\theta^l_j - \theta^{k+1}_{n-1}) \\
&\le \eta \sin(\theta^{k+1}_{\bar{k}_{n}} - \theta^{k+1}_{N_{k+1}}) + \sin (\theta^l_j - \theta^{k+1}_{n-1}) \\
&\le -1 + 1 = 0
\end{aligned} 
\end{equation*}
Hence, combining the above arguments in (a) and (b), we obtain that
\[\eta \sum_{i=n}^{N_{k+1}}\left( \eta^{i-n} \underset{j \le i}{\min_{j\in \mathcal{N}_{i}^{k+1}(k+1)}} \sin (\theta^{k+1}_j(t) - \theta^{k+1}_{i}(t))\right) + \sin (\theta^l_j - \theta^{k+1}_{n-1}) \le 0.\]
$\ $

\noindent \textbf{(iii)} If $\theta^l_j > \theta^{k+1}_{N_{k+1}}$, we exploit the concave property of sine function in $[0, \pi]$ to get
\begin{equation}\label{F-13}
\sin (\theta^l_j - \theta^{k+1}_{n-1}) \le \sin (\theta^l_j - \theta^{k+1}_{N_{k+1}}) + \sin (\theta^{k+1}_{N_{k+1}} - \theta^{k+1}_{n-1}).
\end{equation}
For the second part on the right-hand side of above inequality \eqref{F-13}, we apply the same analysis in (ii) to obtain
\[\eta \sum_{i=n}^{N_{k+1}}\left( \eta^{i-n} \underset{j \le i}{\min_{j\in \mathcal{N}_{i}^{k+1}(k+1)}} \sin (\theta^{k+1}_j - \theta^{k+1}_{i})\right) + \sin (\theta^{k+1}_{N_{k+1}} - \theta^{k+1}_{n-1}) \le 0.\]
For the first part on the right-hand side of \eqref{F-13}, the calculation is the same as \eqref{F-d8}, thus we have
\begin{equation*}
\sin (\theta^{l}_j - \theta^{k+1}_{N_{k+1}}) \le \theta^{l}_j - \theta^{k+1}_{N_{k+1}} \le \theta^{l}_j -  \min_{0 \le i \le k}\min_{1 \le j \le N_i} \{\theta^i_j\} \le D_k(\theta(t)).
\end{equation*}

Therefore, we combine the above estimates to obtain
\begin{equation}\label{F-14}
\begin{aligned}
&\eta S_k \sum_{i=n}^{N_{k+1}}\left( \eta^{i-n} \underset{j \le i}{\min_{j\in \mathcal{N}_{i}^{k+1}(k+1)}} \sin (\theta^{k+1}_j(t) - \theta^{k+1}_{i}(t))\right) + \mathcal{I}_{22} \\
& \le \eta S_k \sin(\theta^{k+1}_{\bar{k}_{n}} - \theta^{k+1}_{N_{k+1}}) + \sum_{l=0}^k \sum_{j\in \mathcal{N}_{n-1}^{k+1}(l)} \sin (\theta^{l}_j - \theta^{k+1}_{n-1}) \\
& \le S_k D_k(\theta(t)).
\end{aligned}
\end{equation}
Then from \eqref{F-8}, \eqref{F-12}, and \eqref{F-14} , it yields that
\begin{equation*}
\begin{aligned}
\frac{d}{dt} \bar{\theta}^{k+1}_{n-1} &\le \Omega_M + S_{k+1}K \sin\alpha + \frac{\bar{a}^{k+1}_{n-1}}{\bar{a}^{k+1}_{n-1}+1} S_kK \cos \alpha D_k (\theta(t)) + + K\cos\alpha \frac{1}{\bar{a}^{k+1}_{n-1} + 1} S_kD_k(\theta(t))\\
&+  K\cos \alpha \frac{1}{\bar{a}^{k+1}_{n-1} + 1} \eta  \sum_{i=n}^{N_{k+1}}\left( \eta^{i-n} \underset{j \le i}{\min_{j\in \mathcal{N}_{i}^{k+1}(k+1)}} \sin (\theta^{k+1}_j - \theta^{k+1}_{i})\right)\\
&+K \cos\alpha \frac{1}{\bar{a}^{k+1}_{n-1}+1} \underset{j \le n-1}{\min_{j\in \mathcal{N}_{n-1}^{k+1}(k+1)}} \sin (\theta^{k+1}_j - \theta^{k+1}_{n-1})\\
&\le \Omega_M + S_{k+1}K\sin\alpha + S_kK\cos\alpha D_k(\theta(t))\\
&+K\cos \alpha \frac{1}{\bar{a}^{k+1}_{n-1} + 1} \sum_{i=n-1}^{N_{k+1}}\left( \eta^{i-(n-1)} \underset{j \le i}{\min_{j\in \mathcal{N}_{i}^{k+1}(k+1)}} \sin (\theta^{k+1}_j - \theta^{k+1}_{i})\right)
\end{aligned}
\end{equation*}
This means that the claim \eqref{F-5} does hold for $n-1$. Therefore, we apply the inductive criteria to verify the claim \eqref{F-5}.\newline 

\noindent $\bigstar$ \textbf{Step 2.} Now we are ready to prove \eqref{F-1} on $J_l$ for Case 2. In fact, we apply Lemma \ref{eta_sin_inequality} and the strong connectivity of $\mathcal{G}_{k+1}$ to have
\begin{equation*}
\sum_{i=1}^{N_{k+1}}\left( \eta^{i-1} \underset{j \le i}{\min_{j\in \mathcal{N}_{i}^{k+1}(k+1)}} \sin (\theta^{k+1}_j - \theta^{k+1}_{i})\right) \le \sin(\theta^{k+1}_1 - \theta^{k+1}_{N_{k+1}})
\end{equation*}
From the notations in \eqref{abbreviation_k} and \eqref{bar_underline_k}, it is known that
\begin{equation*}
\bar{\theta}_1^{k+1} = \bar{\theta}_{k+1}, \quad \underline{\theta}_{N_{k+1}}^{k+1} = \underline{\theta}_{k+1}.
\end{equation*}
Thus, we exploit the above inequality and set $n=1$ in \eqref{F-5} to obtain
\begin{equation}\label{F-15}
\begin{aligned}
\frac{d}{dt} \bar{\theta}_{k+1} &= \frac{d}{dt} \bar{\theta}^{k+1}_1\\
&\le \Omega_M + S_{k+1}K \sin\alpha + S_kK\cos\alpha D_k(\theta(t))\\
&+ K\cos \alpha \frac{1}{\bar{a}^{k+1}_{1} + 1} \sum_{i=1}^{N_{k+1}}\left( \eta^{i-1} \underset{j \le i}{\min_{j\in \mathcal{N}_{i}^{k+1}(k+1)}} \sin (\theta^{k+1}_j - \theta^{k+1}_{i})\right) \\
&\le \Omega_M + S_{k+1}K \sin\alpha + S_kK\cos\alpha D_k(\theta(t)) + K\cos \alpha \frac{1}{\bar{a}^{k+1}_{1} + 1} \sin(\theta^{k+1}_1 - \theta^{k+1}_{N_{k+1}})
\end{aligned}
\end{equation}
We further apply the similar arguments in \eqref{F-15} to derive the differential inequality of $\underline{\theta}_{k+1}$ as below
\begin{equation}\label{F-16}
\frac{d}{dt} \underline{\theta}_{k+1} \ge \Omega_m - S_{k+1}K\sin \alpha - S_kK\cos\alpha D_k(\theta(t)) + K\cos \alpha \frac{1}{\bar{a}^{k+1}_{1} + 1}  \sin (\theta^{k+1}_{N_{k+1}} - \theta^{k+1}_1).
\end{equation}
Due to the monotone decreasing property of $\frac{\sin x}{x}$ in $(0, \pi]$ and from \eqref{F-4}, it is obvious that
\[\sin (\theta^{k+1}_{N_{k+1}} - \theta^{k+1}_1) \ge \frac{\sin \gamma}{\gamma}(\theta^{k+1}_{N_{k+1}} - \theta^{k+1}_{1}).\]
Then we combine the above inequality, \eqref{F-15}, \eqref{F-16} and \eqref{a^k_1-size} to get
\begin{equation*}
\begin{aligned}
\dot{Q}^{k+1}(t) &= \frac{d}{dt} (\bar{\theta}_{k+1} - \underline{\theta}_{k+1}) \\
&\le D(\Omega) + 2S_{k+1}K\sin\alpha + 2S_kK\cos\alpha D_k(\theta(t)) \\
&- K\cos \alpha \frac{2}{\bar{a}^{k+1}_{1} + 1}  \sin (\theta^{k+1}_{N_{k+1}} - \theta^{k+1}_1) \\
&\le D(\Omega) + 2S_{k+1}K\sin\alpha + 2S_kK\cos\alpha D_k(\theta(t)) \\
&- K\cos \alpha \frac{1}{\bar{a}^{k+1}_{1} + 1} \frac{\sin \gamma}{\gamma} (\theta^{k+1}_{N_{k+1}} - \theta^{k+1}_1)\\
&\le D(\Omega) + 2NK\sin \alpha + (2N+1)K \cos \alpha D_k(\theta(t)) \\
&- K\cos \alpha \frac{1}{\sum_{j=1}^{N-1}\eta^jA(2N,j) + 1} \frac{\sin \gamma}{\gamma}Q^{k+1}(t),\qquad t \in J_l,
\end{aligned}
\end{equation*}
where we use the fact that $Q^{k+1}(t) \le \theta^{k+1}_{N_{k+1}}(t) - \theta^{k+1}_1(t)$ and \eqref{a^k_1-size}. Eventually, for Case 2, we obtain the dynamics for $Q^{k+1}(t)$ in \eqref{F-1} on $J_l$, i.e.,
\begin{equation*}
\begin{aligned}
\dot{Q}^{k+1}(t) &\le D(\Omega) + 2NK\sin\alpha + (2N+1)K\cos\alpha D_k(\theta(t)) - \frac{K\cos\alpha}{c}Q^{k+1}(t), \ t\in [0,T^*).\newline
\end{aligned}
\end{equation*}

\end{appendix}


\vspace{0.5cm}

\end{document}